\documentclass{amsart}
\pdfoutput=1
\usepackage{pinlabel}  
\usepackage{amsmath}
\usepackage{amscd}
\usepackage{amssymb}
\usepackage{amsfonts}
\usepackage{amsthm}
\usepackage{color}

\usepackage{enumerate}
\usepackage{graphicx}
\usepackage{caption}
\usepackage{subcaption}
\usepackage{cite}
\usepackage{hyperref}
\usepackage{mathrsfs}
\newcommand{\sphere}{\real/\integer}

\newcommand{\real}{\mathbb{R}}
\newcommand{\integer}{\mathbb{Z}}
\newcommand{\NN}{\mathbb{N}}

\newtheorem{thm}{Theorem}
\newtheorem{cor}{Corollary}

\newtheorem{prop}{Proposition}

\newtheorem{lem}{Lemma}
\theoremstyle{remark}
\newtheorem*{rem}{Remark}
\newtheorem*{Ack}{Acknowledgment}
\theoremstyle{definition}

\newtheorem{defn}{Definition}[section]

\begin{document}
\begin{abstract}    
  The scattering data of a Riemannian manifold with boundary record the incoming and outgoing directions of each geodesic passing through.
  We show that the scattering data of a generic Riemannian surface with no trapped geodesics and no conjugate points determine the lengths of geodesics.
  Counterexamples exists when trapped geodesics are allowed.
\end{abstract}
\author[C.~Croke]{Christopher B. Croke$^+$}
\address{Department of Mathematics, University of Pennsylvania, Philadelphia, PA 19104-6395, USA}
\email{ccroke@math.upenn.edu}
\thanks{$+$ Supported in part by NSF grant DMS 10-03679}

\author[H.~Wen]{Haomin Wen$^{\dagger}$}
\address{Max-Planck-Institut f\"ur Mathematik, Vivatsgasse 7, D-53111 Bonn, Germany}
\email{hwen@mpim-bonn.mpg.de}
\thanks{${\dagger}$ Supported by the Max Planck institute for Mathematics - Bonn}
\title{Scattering data versus lens data on surfaces}

\maketitle
\section{Introduction}
\subsection{Scattering data and lens data}
Let $M$ be a Riemannian manifold.
Let $\pi : \Omega M \rightarrow M$ be the unit tangent bundle of $M$
and $\Omega_x M$ be the set of unit tangent vectors at $x$ for any $x \in M$.
Let $\partial \Omega M$ be the boundary of the unit tangent bundle of $M$.
In other words, $\partial \Omega M = \bigcup_{x \in \partial M} \Omega_x M$.
For each $x \in \partial M$,
let $\nu_M(x)$ be the unit normal vector of $M$ pointing inwards at $x$.
Then put
$\partial_+ \Omega_x M = \{X \in \Omega_x M: (X, \nu_M(x))_{g_M} > 0\}$,
$\partial_0 \Omega_x M = \{X \in \Omega_x M: (X, \nu_M(x))_{g_M} = 0\}$,
and $\partial_- \Omega_x M = \{X \in \Omega_x M: (X, \nu_M(x))_{g_M} < 0\}$.
Also,
write
$\partial_+ \Omega M = \bigcup_{x \in \partial M} \partial_+ \Omega_x M$,
$\partial_0 \Omega M = \bigcup_{x \in \partial M} \partial_0 \Omega_x M$,
and $\partial_- \Omega M = \bigcup_{x \in \partial M} \partial_- \Omega_x M$.

For each $X \in \partial_+ \Omega M$,
there is a geodesic $\gamma_X$ whose initial tangent vector is $X$.
Extend the geodesic as long as possible until it touches the boundary $\partial M$ again.
Put $\tau(X) := \ell(\gamma_X)$, the length of $\gamma_X$.

If the geodesic $\gamma_X$ is of finite length,
call its tangent vector at the other end point $\alpha_M(X)$.
(See Figure \ref{fig:scattering_relation}.)
The map $\alpha_M : \partial_+ \Omega M \rightarrow \partial \Omega M$
defined above is called the \emph{scattering relation} of $M$.
Note that $\alpha_M(X)$ will be undefined if $\gamma_X$ is of infinite length.
\begin{figure}[h]
  \center
  \includegraphics[width=0.25\textwidth]{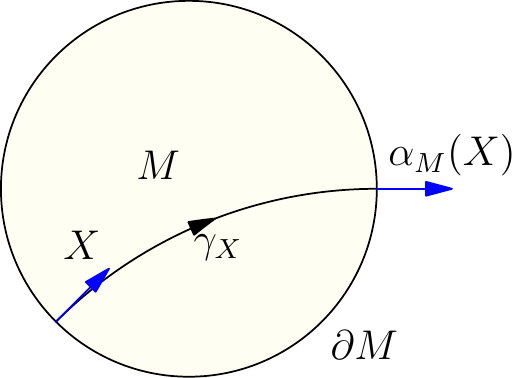}
  \caption{The scattering map $\alpha_M$}
  \label{fig:scattering_relation}
\end{figure}

Suppose that we have two Riemannian manifolds $(M, g_M)$, $(N, g_N)$
and an isometry $h : \partial M \rightarrow \partial N$ between their boundaries.
Then there is a natural bundle map $\varphi : \partial \Omega M \rightarrow \partial \Omega N$ defined as
\begin{align}
  \varphi(a X + b \nu_M(x)) = a h_*(X) + b \nu_N(h(x))
  \label{eq:phi}
\end{align}

For any unit vector $X$ based at $x$ tangent to $\partial M$ and real numbers $a$ and $b$ such that $a^2 + b^2 = 1$.
\begin{defn}
  $M$ and $N$ are said to have the same \emph{scattering data rel $h$} if $\varphi \circ \alpha_M = \alpha_N \circ \varphi$.
  If we also have $\ell(\gamma_X) = \ell(\gamma_{\varphi(X)})$,
  then we say $M$ and $N$ have the same \emph{lens data rel $h$}.
\end{defn}
We will omit ``rel $h$'' when $h$ is clear from the context.

The difference between lens data and scattering data is quite subtle
since lengths of geodesics can be recovered locally from scattering data up to a constant using the first variation of arc length \cite{Mi}.
However, the scattering data do not necessarily determine the lens data completely.
For example, consider the two Riemannian manifolds in Figure \ref{fig:ex1} (which contain trapped geodesics) where the second is obtained from the first by removing a round hemisphere and identifying antipodal points on the boundary great circle.  The two surfaces have the same scattering data but different lens data.  If a geodesic in the first manifold does not enter the hemisphere part then the corresponding geodesic in the second looks the same and has the same length.  However for geodesics that enter the hemisphere the lengths of corresponding geodesics differ by a constant equal to the (intrinsic) diameter of the hemisphere.

\begin{figure}[h]
    \centering
    \begin{subfigure}[b]{0.20\textwidth}
      \centering
      \includegraphics[width=\textwidth]{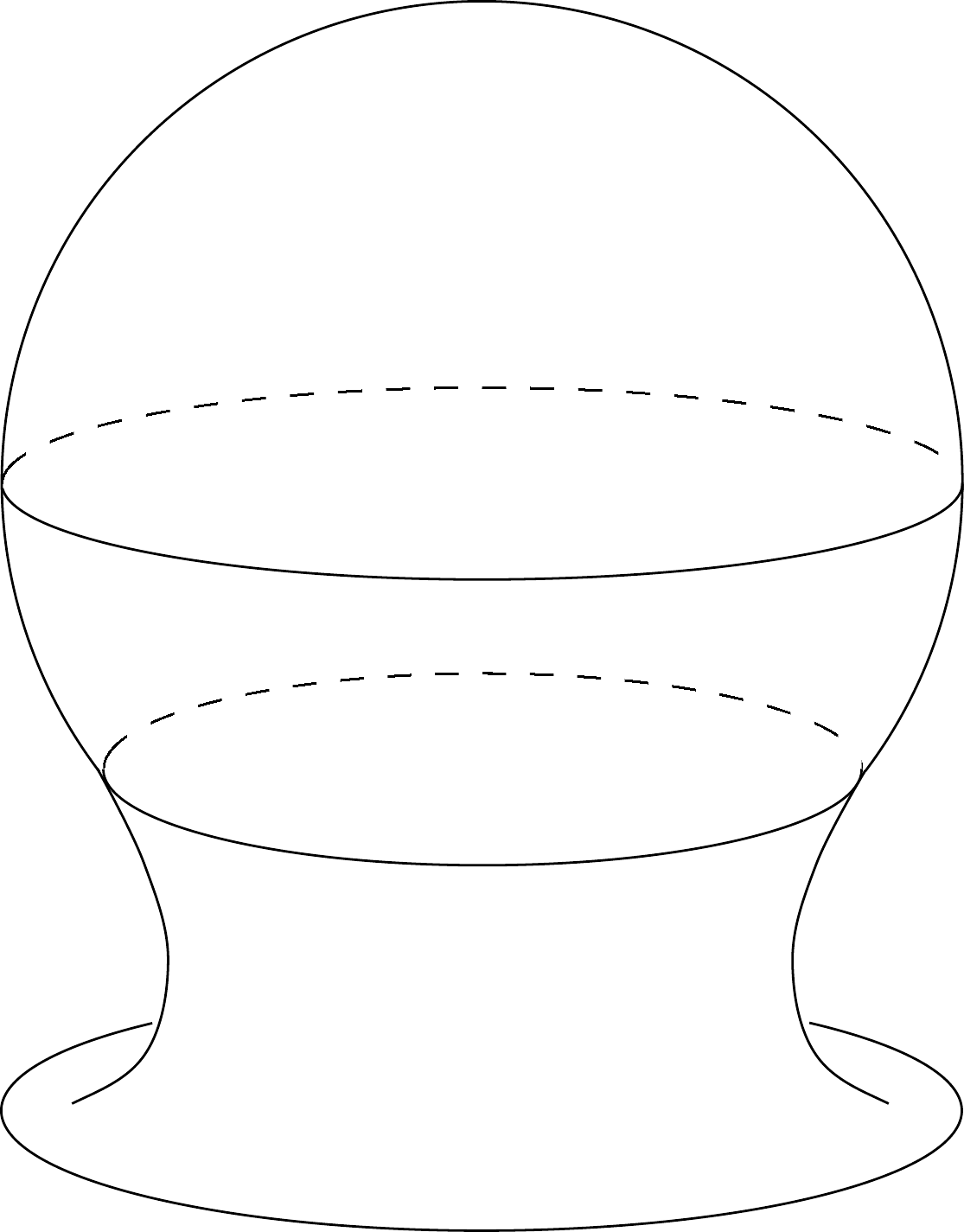}
      \caption{}
      \label{c1}
    \end{subfigure}
    \quad
    \begin{subfigure}[b]{0.20\textwidth}
      \centering
      \includegraphics[width=\textwidth]{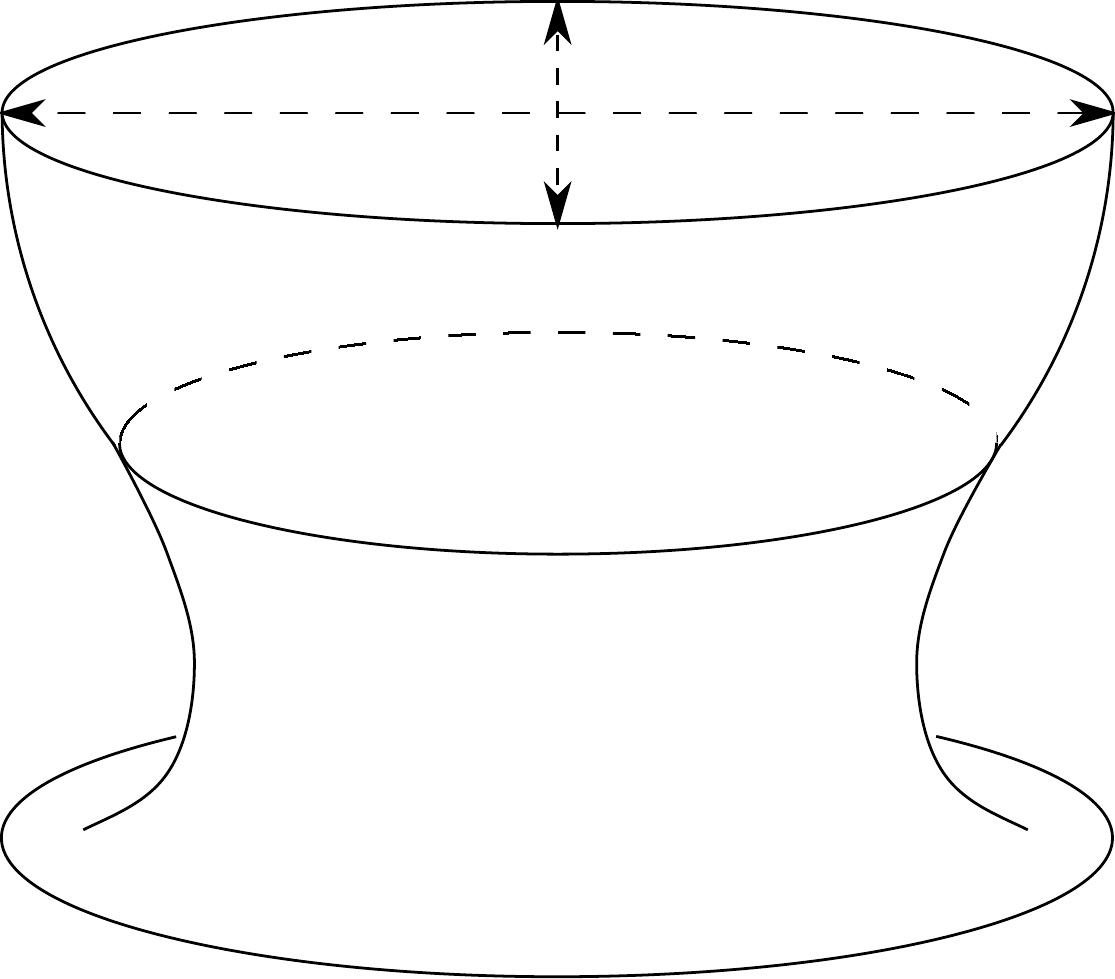}
      \caption{}
      \label{c2}
    \end{subfigure}
    \caption{Same scattering data but different lens data. Here \ref{c2} is obtained from \ref{c1} by removing an upper hemisphere and then identifying antipodal points of the boundary component on the top.}
    \label{fig:ex1}
\end{figure}

\begin{defn}
  $p \in \partial M$ is called a switch point if the geodesic curvature of $\partial M$
  is $0$ at $p$ but not always $0$ in any neighborhood of $p$.
  The set of switch points of $M$ is denoted by $F_M$.
\end{defn}

In general $F_M$ could even be a Cantor set with positive measure, which would create a lot of technical difficulties.  In this paper we will only address the generic case where $F_M$ is finite.

When we say a manifold has no conjugate points we mean that for every
geodesic segment $\gamma:[0,L]\to M$ all nontrivial Jacobi fields can vanish
at most once.
Note that the geodesic here is allowed to be tangent to boundary at points in its interior.
(Also note that this definition is not the same as the one in \cite%
{ABB3}.)

\begin{thm}
  Suppose that we have two compact smooth Riemannian surfaces $(M, g_M)$, $(N, g_N)$
  and an isometry $h : \partial M \rightarrow \partial N$ between their boundaries.
  Assume that $F_M$  is finite, and $M$ has no trapped geodesics (including closed geodesics) and no conjugate points.
  Then $M$ and $N$ have the same scattering data rel $h$ if and only if they have the same lens data rel $h$.
  \label{thm:main}
\end{thm}

\subsection{Scattering rigidity and lens rigidity}

\begin{defn}
  $M$ is \emph{scattering rigid} (resp.\ \emph{lens rigid}) if the space $M$ and the metric on $M$ is determined by its scattering data (resp.\ lens data) up to an isometry which leaves the boundary fixed.
\end{defn}

Showing that a manifold is scattering or lens rigid is an example of a geometric inverse problem (for which there is a vast literature).

A number of manifolds are known to be lens rigid:
\begin{enumerate}
  \item Simple Riemannian surfaces with boundary (L.\ Pestov--G.\ Uhlmann, \cite{pestov-uhlmann-2005})
  \item Compact subdomains of $\real^n$ with flat metrics (M.\ Gromov, \cite{gromov1983filling}) or metrics close to that (D.\ Burago--S.\ Ivanov, \cite{burago2010boundary})
  \item Compact subdomains of open hemispheres (R.\ Michel , \cite{Mi})
  \item Almost hyperbolic metrics (D.\ Burago--S.\ Ivanov, \cite{burago2013area})
  \item Compact subdomains of symmetric spaces of negative curvature (G.\ Besson--G.\ Courtois--S.\ Gallot, \cite{MR1354289})
  \item $D^n \times \sphere^1$ when $n > 1$ (C.\ Croke, \cite{croke2014scattering}) and when $n = 1$ (C.\ Croke--P.\ Herreros, \cite{croke-herreros-2011})

\end{enumerate}
However, very few are known to be scattering rigid:
\begin{enumerate}
  \item
    $D^n \times \sphere^1$ when $n > 1$ (C.\ Croke, \cite{croke2014scattering})
    (It is still not known if the flat annulus $D^1 \times \sphere^1$ is scattering rigid.)
  \item
    Simple Riemannian surfaces with boundary (H. Wen, \cite{We})
\end{enumerate}

Scattering rigidity and lens rigidity are equivalent when the scattering data determine the lens data. Hence we have the following corollary of Theorem \ref{thm:main}

\begin{cor}
  Suppose that
  $M$ satisfies the conditions in Theorem \ref{thm:main},
  then $M$ is scattering rigid if and only if $M$ is lens rigid.
  \label{cor:main}
\end{cor}
\begin{rem}
  Simple Riemannian manifolds \cite{Mi} and, more generally, SGM (strongly geodesically minimizing) manifolds \cite{croke-1991} are conjectured to be lens rigid.
  Most SGM surfaces satisfy our conditions since they have no trapped geodesics and they have conjugate points only in some non-generic cases.
\end{rem}

\section{Space of geodesics}

Geodesics on a smooth Riemannian manifold without boundary satisfy the geodesic equation
and they minimize the length locally.

However, a curve on a smooth Riemannian manifold with boundary that does not satisfy the geodesic equation may still minimize the length locally
if part of the curve runs along the boundary.
\begin{defn}
  A smooth curve $\gamma : [a, b] \rightarrow M$ is called a \emph{geodesic} if it satisfies the geodesic equation $\nabla_{\gamma'}\gamma' = 0$.
  A curve $\gamma$ in $M$ is called a \emph{p-geodesic} if
  it minimize the length locally,
  that is, for any $t \in [a, b]$,
  there is $\delta > 0$
  such that $\gamma|_{[t, t']}$
  is the shortest curve
  connecting $\gamma(t)$ and $\gamma(t')$
  for any $t'$
  such that $|t - t'| < \delta$.
\end{defn}

The basic properties of $p$-geodesics were studied in \cite{ABB1} (also see \cite
{ABB2,ABB3}) (in these references p-geodesics are simply referred to as geodesics).
In particular such a p-geodesic is a $C^1$ path. The path is the union of
not only interior segments (smooth geodesic segments in the usual sense) and
boundary segments (smooth geodesic segments of the boundary) and switch
points (where the path joins two of the previous types and is not twice
differentiable) but also accumulation points of switch points, called
intermittent points. On can even have cantor sets of positive measure of
intermittent points. The boundary will have 0 curvature (as will the $p$%
-geodesic) at intermittent points. In two dimensions it is also easy to see
that if $\gamma(t)\in \partial M$ then the boundary is concave at $\gamma(t)$
(though maybe not strictly) since otherwise one could locally shorten the
curve.

In section \ref{NCP} we will prove:

\begin{prop}
\label{minimizing}

Let $M$ be a compact 2-manifold (with or without boundary) that has no
conjugate points and finite $F_M$, and $\gamma$ a p-geodesic segment between points $x\in M$
and $y\in M$. Then for any curve $\tau$ from $x$ to $y$ homotopic to $\gamma$
(relative to $x$ and $y$) we have $L(\gamma)\leq L(\tau)$.  Further $L(\gamma)=L(\tau)$ only when $\gamma=\tau$ up to parametrization.
\end{prop}

We should remark that that the assumption that $F_M$ is finite in Proposition \ref{minimizing} is probably unnecessary.  In particular it is easy to see (and is of independent interest):

\begin{cor}
Let $M$ be a compact 2-manifold (with or without boundary) that has no
conjugate points, and $\gamma$ a geodesic segment between points $x\in M$
and $y\in M$. Then for any curve $\tau$ from $x$ to $y$ homotopic to $\gamma$
(relative to $x$ and $y$) we have $L(\gamma)\leq L(\tau)$.  Further $L(\gamma)=L(\tau)$ only when $\gamma=\tau$ up to parametrization.
\end{cor}

\begin{proof}
We can extend our metric smoothly to a metric on $M$ union a collar neighborhood of the boundary. We can then
change our boundary by an arbitrarily small amount to an embedded curve that
has a finite $F_M$ and lies totally in the collar neighborhood. If the
perturbation is small enough geodesics in the expanded space will still have
no conjugate points so if the Proposition is true when $F_M$ is finite then
each geodesic segment $\gamma$ minimizes in its homotopy class in the larger
space and hence in the original space (since the perturbation can be made so
as to preserve homotopy classes).
\end{proof}

Let $\tilde{M}$ be the universal cover of $M$.
Let $\Gamma^p_{\tilde{M}}$ be the
space of p-geodesics $[0, 1] \rightarrow \tilde{M}$
with the compact open topology.
Define $(s_M, e_M) : \Gamma^p_{\tilde{M}} \rightarrow \tilde{M} \times \tilde{M}$
as $s_M(\gamma) = \gamma(0)$ and $e_M(\gamma) = \gamma(1)$.

Another easy consequence of Proposition \ref{minimizing} is:

\begin{prop}
\label{producthomeo}
  If ${M}$ has no conjugate points and finite $F_M$, then $(s_M, e_M)$ is a homeomorphism.
  \label{prop:no_conj}
\end{prop}

\section{non-convex part of the boundary}
In the rest of the paper, $M$ and $N$ will be two compact smooth Riemannian surfaces with the same scattering data rel $h : \partial M \rightarrow \partial N$
where $h$ is an isometry.
$M$ is assumed to have finitely many switch points, no trapped geodesics (including closed geodesics) and no conjugate points.
$\varphi : \partial \Omega M \rightarrow \partial \Omega N$
is the induced bundle map defined in \eqref{eq:phi}.

We say that $\partial M$ is strictly convex near $p$ if the curvature of $\partial M$ is positive at $p$.
We say that $\partial M$ is strictly concave near $p$ if the curvature of $\partial M$ is negative at $p$.
We say that $\partial M$ is totally geodesic near $p$ if the curvature of $\partial M$ is zero \emph{near} $p$.
Let $S_- = \{p \in \partial M : \text{$\partial M$ is strictly concave near $p$}\}$,
$S_+ = \{p \in \partial M : \text{$\partial M$ is strictly convex near $p$}\}$
and
$S_0 = \{p \in \partial M : \text{$\partial M$ is totally geodesic near $p$}\}$.
Here being strictly concave means that the curvature is negative,
Note that $S_-$, $S_+$ and $S_0$ are all open in $\partial M$
and that $S_s := M \setminus (S_- \bigcup S_+ \bigcup S_0)$ is the set of switch points.

\begin{prop}
  If $p_0 \in \overline{S_-}$,
  then the curvature of $\partial M$ at $p_0$ is the same as the curvature of $\partial N$ at $h(p_0)$.
  \label{prop:concave}
\end{prop}
\begin{proof}
  If $p_0 \in S_-$,
  there is a $p \in S_-$ near $p_0$ such that there is a geodesic in $M$ which is tangent to $\partial M$ at $p$, which intersects $\partial M$ transversely at the two end points, and which have no other intersections with $\partial M$.
  Since $M$ and $N$ have the same scattering data, the same thing happens to $h(p)$.
  Hence the $C^\infty$ jet of the metric near $h(p)$ is determined \cite{UW}
  by the scattering data.
  In particular,
  the curvature of $S_-$ at $p$ is the same as the curvature of $h(S_-)$ at $h(p)$.
  Since $M$ and $N$ are assume to be smooth,
  the curvature of $S_-$ at $p_0$ is the same as the curvature of $h(S_-)$ at $h(p_0)$.
\end{proof}

\begin{prop}
  If $p_0 \in S_0$,
  then $\partial N$ is totally geodesic at $h(p_0)$.
  \label{prop:totally_geodesic}
\end{prop}
\begin{proof}
  Let $U$ be an open neighborhood of $p_0$ in $S_0$.
  Pick a unit tangent vector $X_0 \in \partial_0 \Omega_{p_0} M$.
  There are two choices of $X_0$ but either will work.

  For any $\theta \in (0, \pi)$, let $X_\theta$ be the unit tangent vector in $\partial_+ \Omega_{p_0} X$ such that the angle between $X_0$ and $X_\theta$ is $\theta$.
  We shall show that
  there is a $\delta > 0$ such that
  $\gamma_{X_\theta}$ is not tangent to $\partial M$ when $\theta \in (0, \delta)$.
  Suppose that this is not true,
  then there is a monotonically decreasing sequence $\theta_i \rightarrow 0$ such that
  $\gamma_{X_{\theta_i}}$ is tangent to $\partial M$.
  Since $M$ has no trapped geodesics,
  the lengths of geodesics in $M$ is bounded from above universally.
  Hence there is a subsequence $\theta_{i_k}$ of $\theta_i$ such that
  $\gamma_{X_{\theta_{i_k}}}$ is tangent to $\partial M$ at $q_k = \gamma_{X_{\theta_{i_k}}}(s_k)$
  where $s_k$ converges.
  Now, lift $p_0$ to a point $\tilde{p}_0$ in $\tilde M$, the universal cover of $M$.
  Then lift each $X_\theta$ to a unit tangent vector $\tilde{X}_\theta \in \Omega_{\tilde{p}_0} \tilde{M}$.
  Let $\tilde{q}_k = \gamma_{\tilde{X}_{\theta_{i_k}}}(s_k)$,
  then $\tilde{q}_k$ converges to some $\tilde{q} \in \partial \tilde{M}$.
  So there is $N \in \NN$ such that $\tilde{q}_k$ and $\tilde{q}$ are on the same component of $\partial \tilde{M}$ when $k > N$.
  It follows that there is switch point between $\tilde{q}_k$ and $\tilde{q}_{k+1}$ when $k > N$,
  which contradicts our assumption that $F_M$ is finite.

  Define $\gamma_\theta : [0, 1] \rightarrow M$ as
  $\gamma_\theta(t) = \gamma_{X_\theta}(\ell(\gamma_{X_\theta}) t)$.
  We will define $\gamma_0$ as the limit of $\gamma_\theta$ as $\theta \rightarrow 0$,
  if the limit exist.
  By Proposition \ref{prop:no_conj}, $\gamma_\theta$ converges as $\theta \rightarrow 0$ if
  $\gamma_\theta(1)$ converges.
  Let $S_1$ be the component of $\partial M$ that contains $\gamma_{\frac{\delta}{2}}(1)$.
  The no conjugate points condition says that $\gamma_\theta(1)$ moves in a fixed direction on $S_1$ as $\theta$ goes to $0$.
  Assume that $\gamma_\theta(1)$ does not converge as $\theta \rightarrow 0$,
  then $\gamma_{\theta}(1)$ must go around $S_1$ infinitely often.
  Pick any $q \in S_1$,
  there a sequence of positive numbers $\theta_i \rightarrow 0$ such that
  $\gamma_{\theta_i}(1) = q$.
  Substituting $\theta_i$ by a subsequence if necessary,
  we may assume that
  $\gamma_{\theta_i}'(1) / \tau(X_{\theta_i})$ converges to a unit tangent vector $Y$.
  Then $\gamma_{-Y}$ is tangent to $\partial M$ at $p_0$.
  Therefore, there is geodesic in $M$ which goes through $q$ and which is tangent to $\partial M$ at $p_0$.
  However, this is impossible since $q$ is an arbitrary point on $S_1$.
  Hence $\gamma_\theta(1)$ converges as $\theta \rightarrow 0$.
  Therefore, $\gamma_\theta$ converges to $\gamma_0$ as $\theta \rightarrow 0$.

  Since $M$ and $N$ have the same scattering data,
  $\gamma_{\varphi(X_\theta)}$ also converges to a geodesic ray $\gamma_1$ whose initial tangent vector is $\varphi(X_0)$.
  Since $p_0$ is in the interior of $S_0$,
  $\ell(\gamma_0) > 0$.
  Since $M$ has no closed geodesics,
  $X_0 \neq \gamma_0'(1) / \ell(\gamma_0)$.
  Hence $\varphi(X_0) \neq \lim_{\theta \rightarrow 0} \varphi(\alpha(X_\theta))$,
  which implies that $\ell(\gamma_1) > 0$.
  We have $\gamma_1'(1) / \ell(\gamma_1) = \lim_{\theta \rightarrow 0} \varphi(\alpha(X_\theta))
  = \varphi(\lim_{\theta \rightarrow 0} \alpha(X_\theta)) = \varphi(\gamma_0'(1) / \ell(\gamma_0))$.
  Since $\partial M$ is totally geodesic near $p_0$,
  there is $\delta > 0$ such that
  $\gamma_0|_{[0, \delta]} \subset \partial M$.
  For any $p \in \gamma_0((0, \delta))$,
  since $\partial M$ is totally geodesic near $p_0$,
  there is $\delta > 0$ such that
  $\gamma_0|_{[0, \delta]} \subset \partial M$.
  Using the same argument as above,
  we can show that there is
  a geodesic $\gamma_3$ in $N$ starting at $h(p)$ such that
  $\gamma_3'(1) / \ell(\gamma_3) = \varphi(\gamma_0'(1) / \ell(\gamma_0))$,
  which implies that $h(p)$ is also on $\gamma_2$.
  Hence $h(\gamma_0([0, \delta]))$ is totally geodesic in $N$.

  There is geodesic ray $\gamma_4$ in $M$ whose initial tangent vector is $-X_0$.
  Using the same argument as above,
  there is a small $\delta' > 0$ such that $\gamma_4([0, \delta]) \subset \partial N$ and that $h(\gamma_4([0, \delta]))$ is totally geodesic.
  Therefore, $\partial N$ is totally geodesic near $h(p_0)$.

\end{proof}

\section{Space of geodesics, continued}
Let $\Gamma'_M$ be the space of maximal geodesics which are tangent to $\partial M$ at a switch point.
We may assume that geodesics in $\Gamma'M$ are not tangent to $S_-$.
If any geodesic in $\Gamma'M$ is tangent to $S_-$,
we may extend $M$ near that tangent point in a collar neighborhood to reduce the number of such tangent points without introducing new switch points.
Since there are only finite many switch points,
all such tangent points can be eliminated by extending $M$ to a new manifold $M'$.
If the extension is small enough, the $M'$ will also have no trapped geodesics and no conjugate points.
Now, do the same extension for $N$.
Namely, glue $N$ and $M' \setminus M$ to obtain an extension $N'$ of $N$.
By proposition \ref{prop:concave},
$N'$ is $C^2$ (actually $C^\infty$).
Since $M$ and $N$ have the same scattering data,
$M'$ and $N'$ also have the same scattering data.
If $M'$ and $N'$ have the same lens data,
then $M$ and $N$ also have the same lens data.
Hence it suffices to prove that
$M'$ and $N'$ have the same lens data.
Thus, without loss of generality,
we assume that geodesics in $\Gamma'_M$ are not tangent to $S_-$.

\subsection{A map between space of geodesics}
Recall that $\Gamma^p_M$ is the space of p-geodesics $[0, 1] \rightarrow M$ whose end points are on $\partial M$, and we define $\Gamma^p_N$ similarly.
For our convenience, any reparametrization of a p-geodesic $\gamma : [0, 1] \rightarrow M$ whose end points are on $\partial M$ will also be viewed as $\gamma \in \Gamma^p_M$.
Let $\Gamma^0_M := \{\gamma \in \Gamma^p_M : \text{$\gamma$ is non-constant and not tangent to $\partial M$ at points in $\partial S_+$.}\}$.

\begin{prop}
  There is a map $\Phi : \Gamma^0_M \rightarrow \Gamma^p_N$
   which satisfies the following conditions.
  \begin{enumerate}
    \item $\Phi$ is continuous with respect to the compact open topology.
    \item For any $\gamma(t) \in \partial M$, reparametrizing $\Phi(\gamma)$ if necessary, we have $\Phi(\gamma)(t) = h(\gamma(t))$
      and $\frac{\Phi(\gamma)'(t)}{|\Phi(\gamma)'(t)|} = \varphi(\frac{\gamma'(t)}{|\gamma'(t)|})$.
  \end{enumerate}
\end{prop}
\begin{rem}
  We will extend $\Phi$ to $\Gamma^p_M$ in the next section.
\end{rem}

\begin{proof}
  Let $\Gamma^1_M$ be the space of non-constant geodesics $[0, 1] \rightarrow M$ whose end points are on $\partial M$ and which are not tangent to $\partial M$ anywhere.

  Let $\Gamma^c_M$ be the space of constant geodesics on $\partial M$,
  and define $\Gamma^2_M := \overline{\Gamma_M^1} \setminus (\Gamma^c_M \bigcup \Gamma'_M)$.

  Let $\Gamma^b_M$ be the space of non-constant p-geodesics in $M$ which run along $\overline{S_- \bigcup S_0} \setminus \partial S_+$, the non-convex part of $\partial M$.

  Defined $\Phi : \Gamma^1_M \rightarrow \Gamma^p_N$ as
  $\Phi(\gamma_X) = \gamma_{\varphi(X)}$.
  $\Phi$ is obviously continuous.
  Next, extend $\Phi$ to a continuous map from $\Gamma^2_M$ to $\Gamma^p_N$ by taking limits.
  If a geodesic $\gamma = \lim_{i \rightarrow \infty} \gamma_i$ is in $\Gamma^2_M$ where $\gamma_i \in \Gamma^1_M$,
  then define $\Phi(\gamma) := \lim_{i \rightarrow \infty} \Phi(\gamma_i)$.
  This is well-defined because $M$ and $N$ have the same scattering data.

  Note that
  we exclude constant curves and geodesics tangent to switch points when defining $\Gamma^2_M$
  because
  $\lim_{i \rightarrow \infty} \Phi(\gamma_i)$
  would not necessarily converge if $\gamma$ were a constant curve or a geodesic in $\Gamma'_M$.
  For example,
  let $\gamma$ be a constant curve and $\gamma^1_i$ a sequence of geodesics in $\Gamma^1_M$ such that
  $\gamma^1_i(0) = \gamma(0)$ and that
  $\gamma^1_i(1)$ converges to $\gamma(0)$ from one side.
  Then $\lim_{i \rightarrow \infty} \Phi(\gamma^1_i)$
  might be a closed geodesic in $N$ tangent $\partial M$ at $h(\gamma(0))$.
  Let $\gamma^2_i$ a sequence of geodesics in $\Gamma^1_M$ such that
  $\gamma^2_i(0) = \gamma(0)$ and that
  $\gamma^2_i(1)$ converges to $\gamma(0)$ from the other side.
  Then $\lim_{i \rightarrow \infty} \Phi(\gamma^2_i)$ will be a closed geodesic
  which goes in the other direction.
  We will resolve this issue by showing that there are no closed geodesics in $N$ tangent to $\partial N$ in the next section.

  For each $\gamma \in \Gamma^b_M$,
  $h \circ \gamma$ is also a p-geodesic
  by Proposition \ref{prop:concave} and Proposition \ref{prop:totally_geodesic}.
  Hence we may define $\Phi(\gamma) := h \circ \gamma$.

  Now, pick any $\gamma \in \Gamma^0_M$.
  Since $F_M$ is finite, we have a decomposition
  $\gamma = \gamma_1 * \gamma_2 * \dots * \gamma_n$
  where each $\gamma_k$ is either in $\Gamma^2_M$ or in $\partial M$.
  Define $\Phi(\gamma) = \Phi(\gamma_1) * \Phi(\gamma_2) * \dots * \Phi(\gamma_n)$.
\end{proof}

Define $e : \Gamma^0_M \rightarrow \real$ as
$e(\gamma) = \ell(\Phi(\gamma)) - \ell(\gamma)$.
The two manifolds will have the same length data if $e = 0$.
\begin{prop}
  $e$ is constant on each component of $\Gamma^0_M$
  \label{prop:constante}
\end{prop}
\begin{proof}
  This follows from the assumption that $M$ and $N$ have the same scattering data. Since the first variation of arc length of geodesics in $\Gamma^p_M$ and $\Gamma^p_N$ only depends on the scattering data, $e$ is constant on each component of $\Gamma^0_M$.
\end{proof}

\subsection{Pairs of p-geodesics}
Recall that $\Gamma^p_M$ is the space of p-geodesics $[0, 1] \rightarrow M$ whose end points are on $\partial M$.
For any $p, q \in \partial M$,
define $P_M(p, q) := \{(\gamma_1, \gamma_2) \in \Gamma^p_M \times \Gamma^p_M : \gamma_1(0) = p, \gamma_2(0) = q, \gamma_1(1) = \gamma_2(1)\}$.
Each element in $P_M(p, q)$ will be called a \emph{pair} (of geodesics based at $p$ and $q$),
and the end point $\gamma_1(1)$ will be called the \emph{root} of the pair, denoted by $r(\gamma_1, \gamma_2)$.

Put $P_M^0(p, q) = P_M(p, q) \bigcap (\Gamma^0_M \times \Gamma^0_M)$.
Then define $l : P_M^0(p, q) \rightarrow \real$ as
$l(\gamma_1, \gamma_2) = e(\gamma_1) - e(\gamma_2)$.
Since $e$ is constant on each component of $\Gamma^0_M$,
$l$ is constant on each component of $P_M^0(p, q)$.

Let $Q_M(p, q) \subset P_M(p, q)$ be the space of pairs which do not overlap at the end.
($\gamma_1$ and $\gamma_2$ are said to overlap at the end if there are $a_1 \in [0, 1)$ and $a_2 \in [0, 1)$ such that $\gamma_1|_{[a_1, 1]}$ and $\gamma_2|_{[a_2, 1]}$ coincide.)
Suppose that $(\gamma_1, \gamma_2) \in P_M(p, q)$ and that $\gamma_1$ and $\gamma_2$
has an overlapping part $\gamma_3$ at the end.
If we remove $\gamma_3$ from $\gamma_1$ and $\gamma_2$,
then we obtain another pair $(\gamma_4, \gamma_5) \in P_M$ where
$\gamma_1 = \gamma_4 * \gamma_3$ and $\gamma_2 = \gamma_5 * \gamma_3$.
Define $b : P_M \rightarrow Q_M$ as
$b(\gamma_1, \gamma_2) = (\gamma_4, \gamma_5)$.
View $b$ as a quotient map and use the quotient topology on $Q_M$ induced by the compact open topology on $P_M$.

Since $e(\gamma_1) = \ell(\gamma_1) - \ell(\Phi(\gamma_1)) = \ell(\gamma_4) + \ell(\gamma_3) - \ell(\Phi(\gamma_4)) - \ell(\Phi(\gamma_3)) = e(\gamma_4) + e(\gamma_3)$,
we have
\begin{align*}
  l(\gamma_1, \gamma_2) &= e(\gamma_1) - e(\gamma_2)\\
  &= e(\gamma_4) + e(\gamma_3) - e(\gamma_5) - e(\gamma_3)\\
  &= e(\gamma_4) - e(\gamma_5)\\
  &= l(\gamma_4, \gamma_5)\\
  &= l (b (\gamma_1, \gamma_2)).
\end{align*}
Hence $l$ is also well-defined on $Q^0_M(p, q) := b(P^0_M(p, q))$,
and $l$ is constant on each component of $Q^0_M(p, q)$.
\begin{prop}
  Each component of $Q_M(p, q)$ is a 1-manifold without boundary.
\end{prop}

\begin{proof}
  When the two geodesics are not tangent to each other at the root, nearby pairs in $Q_M(p,q)$ are determined by their roots. Hence a neighborhood of this pair is homeomorphic to a neighborhood of its root in the boundary, which is a 1-manifold. (See Figure \ref{fig:jump0}.)
  \begin{figure}[h]
    \centering
    \includegraphics[width=0.2\linewidth]{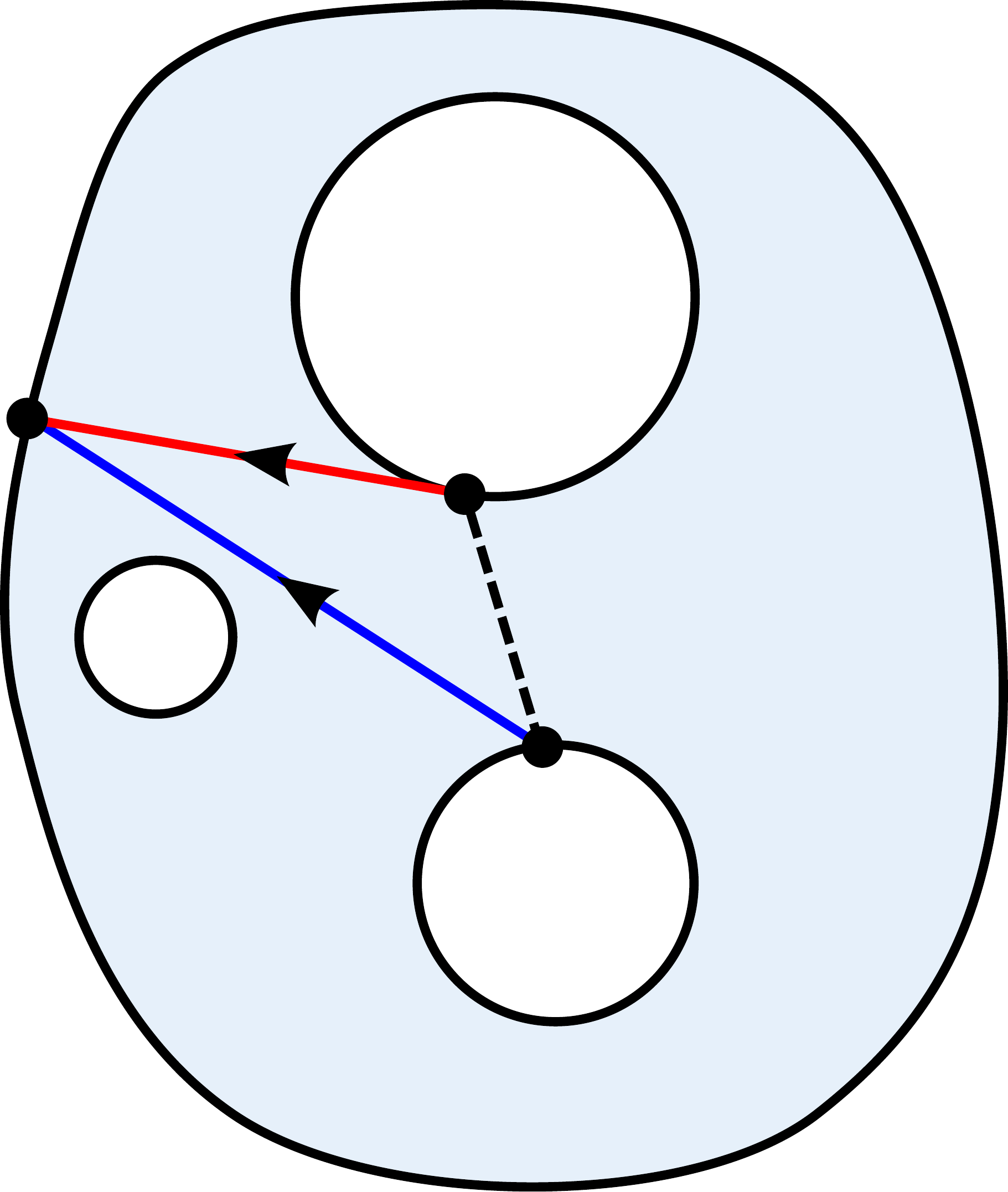}
    \caption{The two geodesic are not tangent to each other at the root.}
    \label{fig:jump0}
  \end{figure}

  If we move the root along the boundary, we may get a pair which has an overlapping part like the one in Figure \ref{fig:jump3}.
  Note that this pair is not in $Q_M(p, q)$ since it has an overlapping part.
  \begin{figure}[h]
    \centering
    \includegraphics[width=0.2\linewidth]{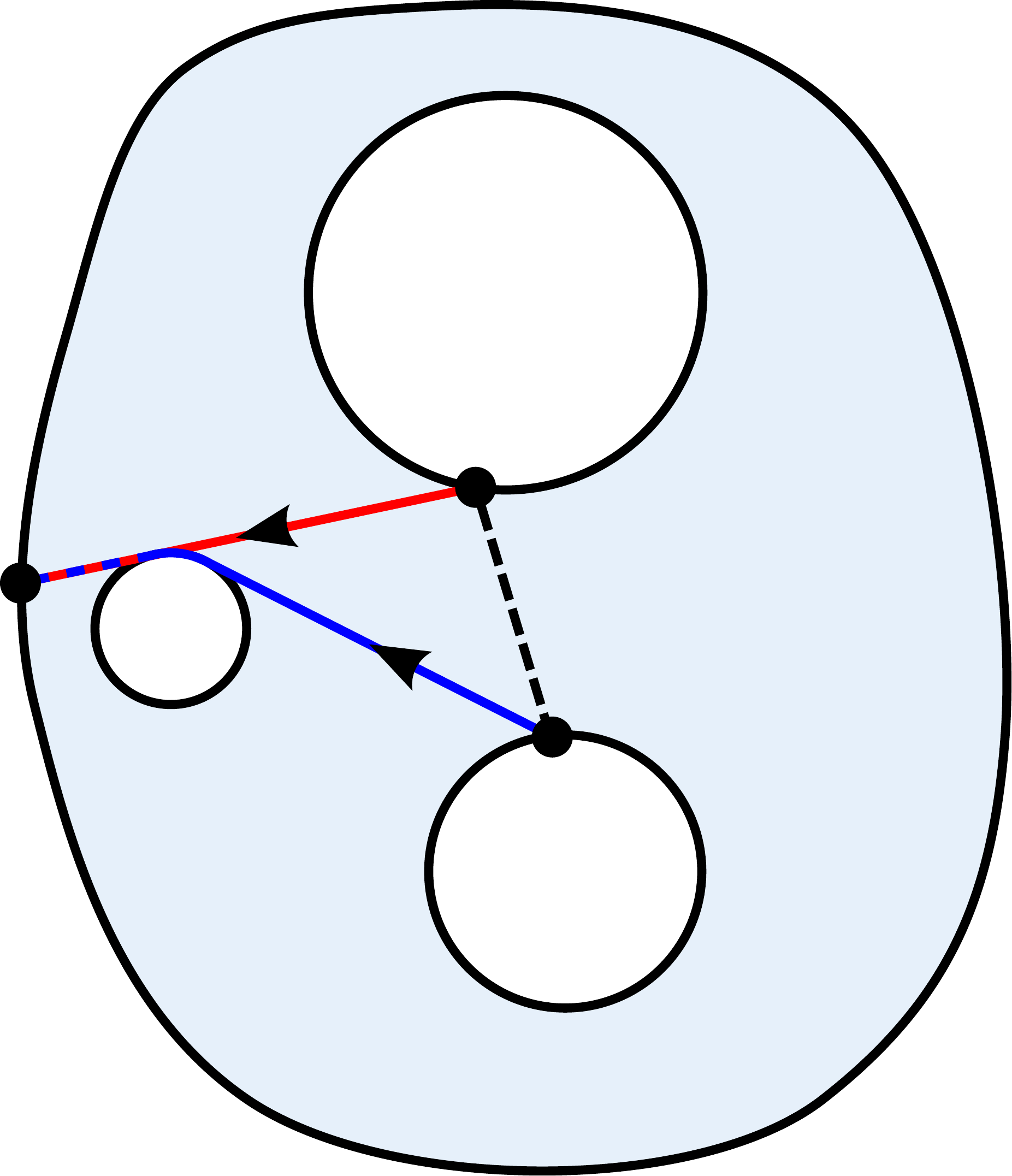}
    \caption{The two geodesic in a pair have an overlapping part.}
    \label{fig:jump3}
  \end{figure}

  The overlapping part will persist if we move the root further (Figure \ref{fig:jump4}).
  \begin{figure}[h]
    \centering
    \includegraphics[width=0.2\linewidth]{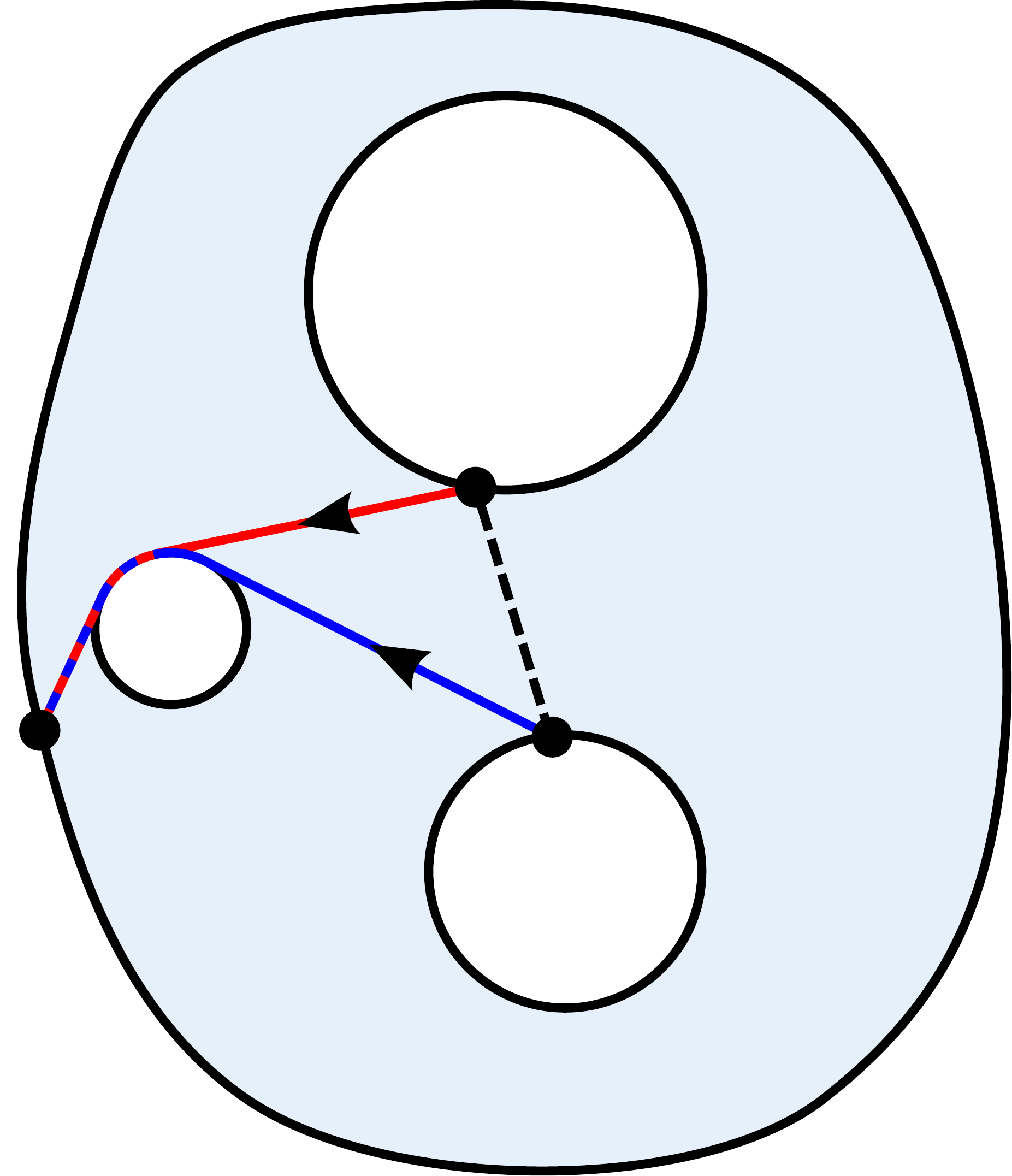}
    \caption{Not in $Q_M(p, q)$.}
    \label{fig:jump31}
  \end{figure}
  Such pairs are not in $Q_M(p, q)$ since they have overlapping part at the end.

  $b$ of the pair in Figure \ref{fig:jump3} gives us a pair in $Q_M(p,q)$ (Figure \ref{fig:jump4}).
  This pair will be called $q_1$.
  \begin{figure}[h]
    \centering
    \includegraphics[width=0.2\linewidth]{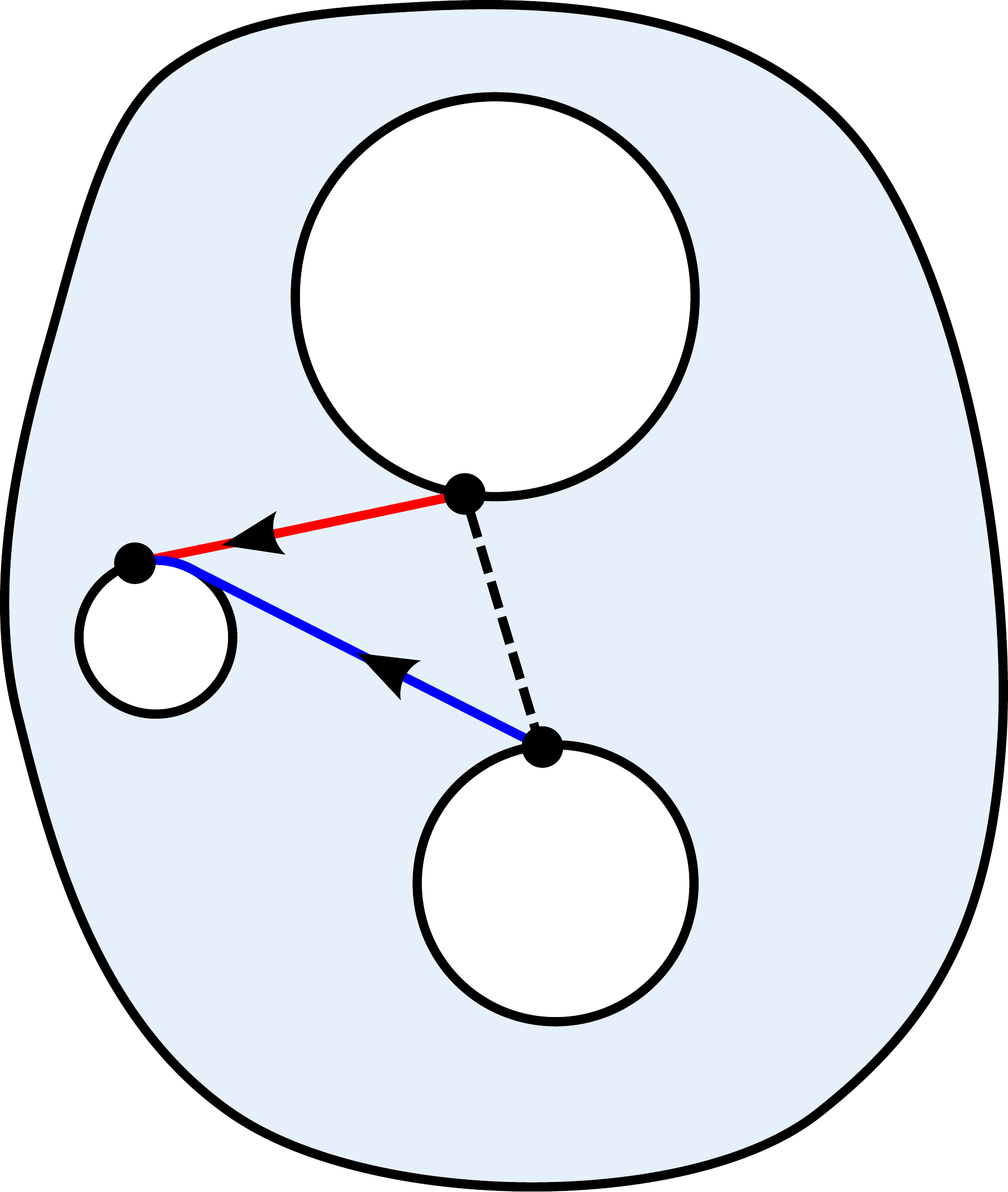}
    \caption{$q_1$}
    \label{fig:jump4}
  \end{figure}

  Now, we can move the root in two directions, but only one of them gives elements in $Q_M(p,q)$.
  For example, in Figure \ref{fig:jump4},
  if we move the root to the left, the pair will have an overlapping part
  while moving the root to the right gives pairs in $Q_M(p, q)$ (Figure \ref{fig:jump5}).
  \begin{figure}[h]
    \centering
    \includegraphics[width=0.2\linewidth]{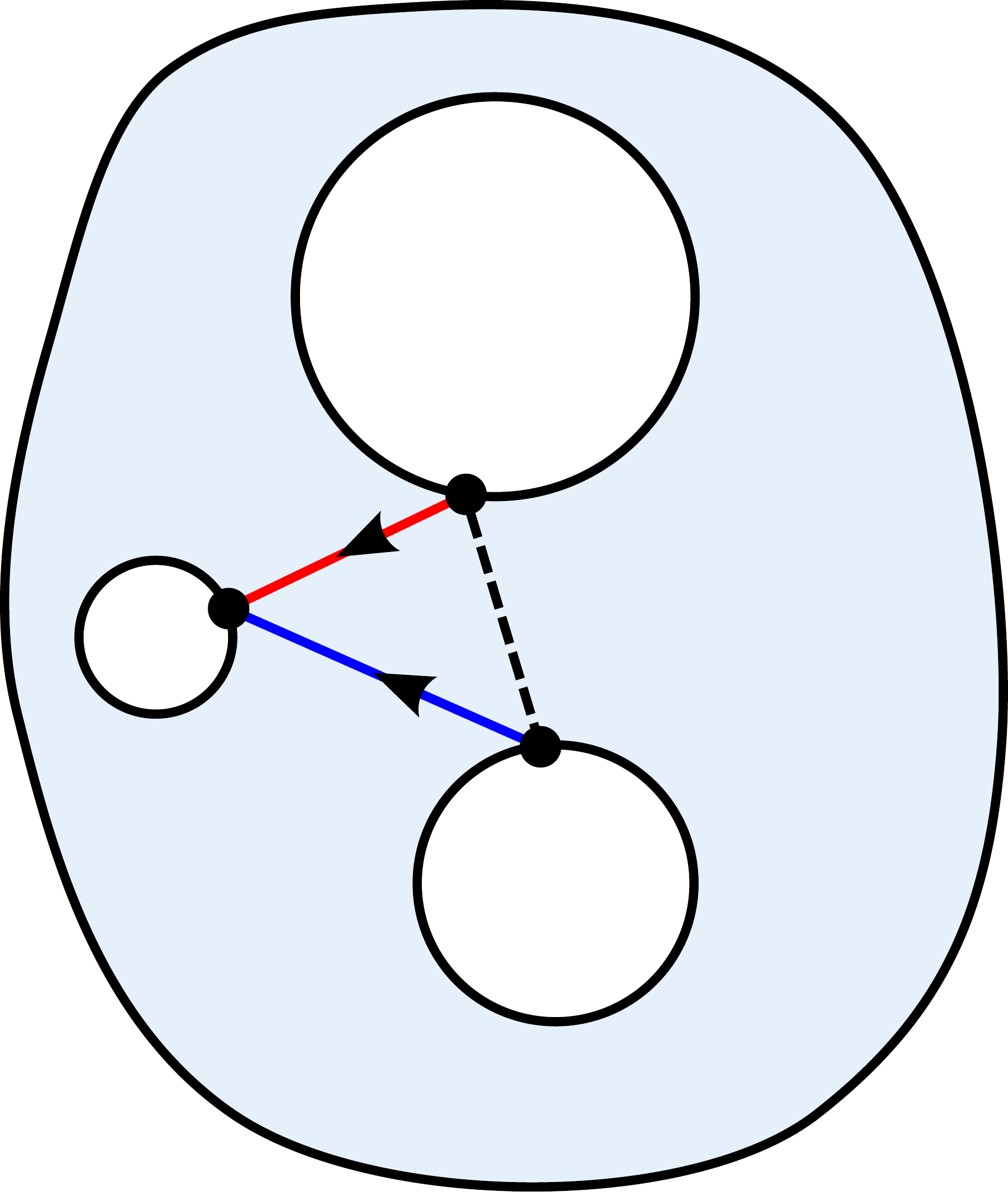}
    \caption{}
    \label{fig:jump5}
  \end{figure}

  So far, we have deformed the pair in Figure \ref{fig:jump0} to the pair Figure \ref{fig:jump4} by moving the root continuously except when going from Figure \ref{fig:jump3} to Figure \ref{fig:jump4}.
  Denote the pairs in between (except the one in Figure \ref{fig:jump3}) by $U$. $U$ is actually a neighborhood of $q_1$ in $Q_M(p, q)$.
  $q_1$ separates $U$ into two parts,
  and each part is a  1-manifold since roots in each part are taken from a 1-manifold continuously.

  Now, consider any pair $q_2$ consisting of geodesics which are tangent to each other at the root (Figure \ref{fig:q2}).
  \begin{figure}[h]
    \centering
    \includegraphics[width=0.2\linewidth]{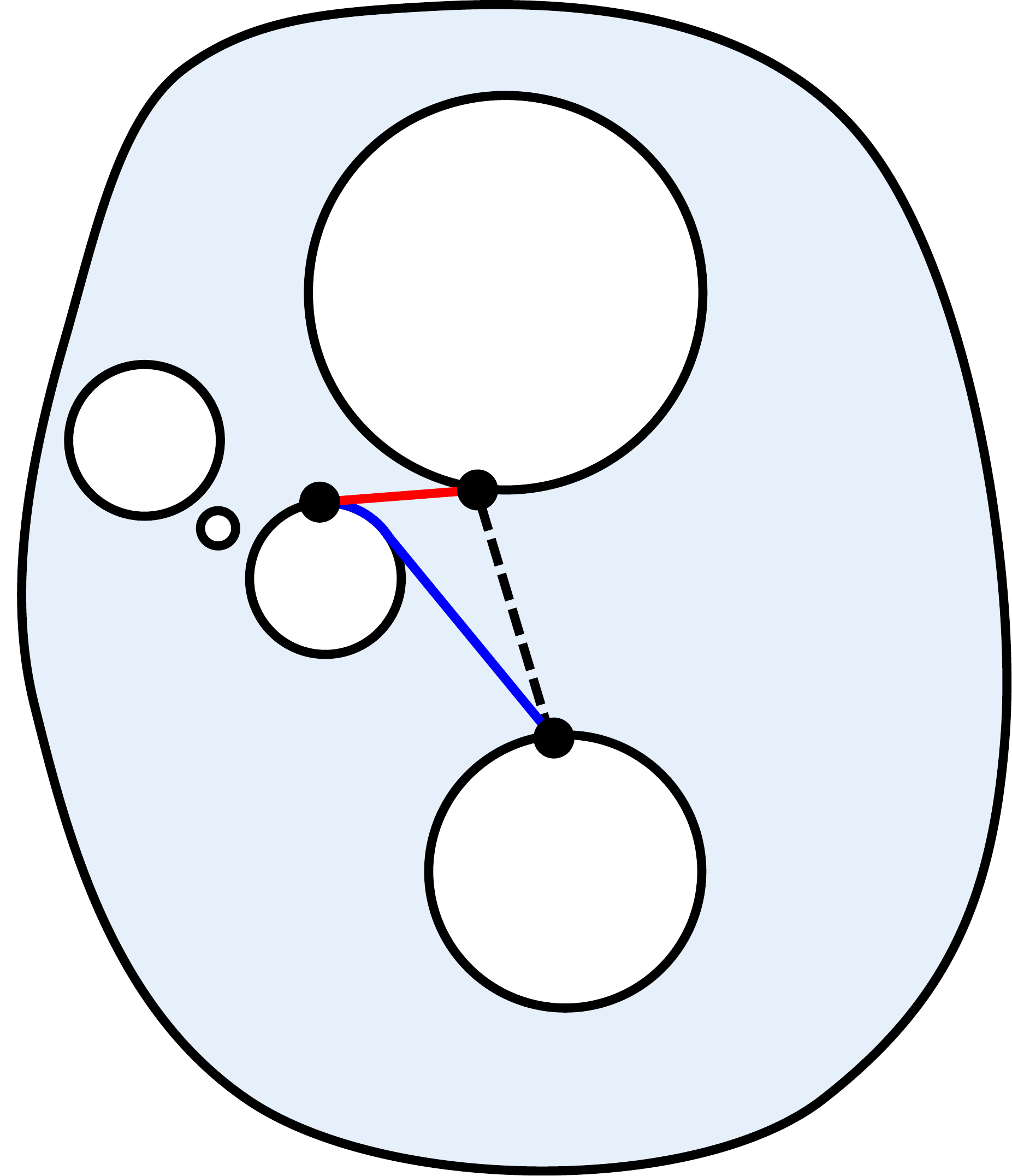}
    \caption{$q_2$}
    \label{fig:q2}
  \end{figure}
  Extending the geodesics in $q_2$ as long as possible,
  we obtain a new pair $q_3 \in P_M(p, q)$ (Figure \ref{fig:q3}).
  \begin{figure}[h]
    \centering
    \includegraphics[width=0.2\linewidth]{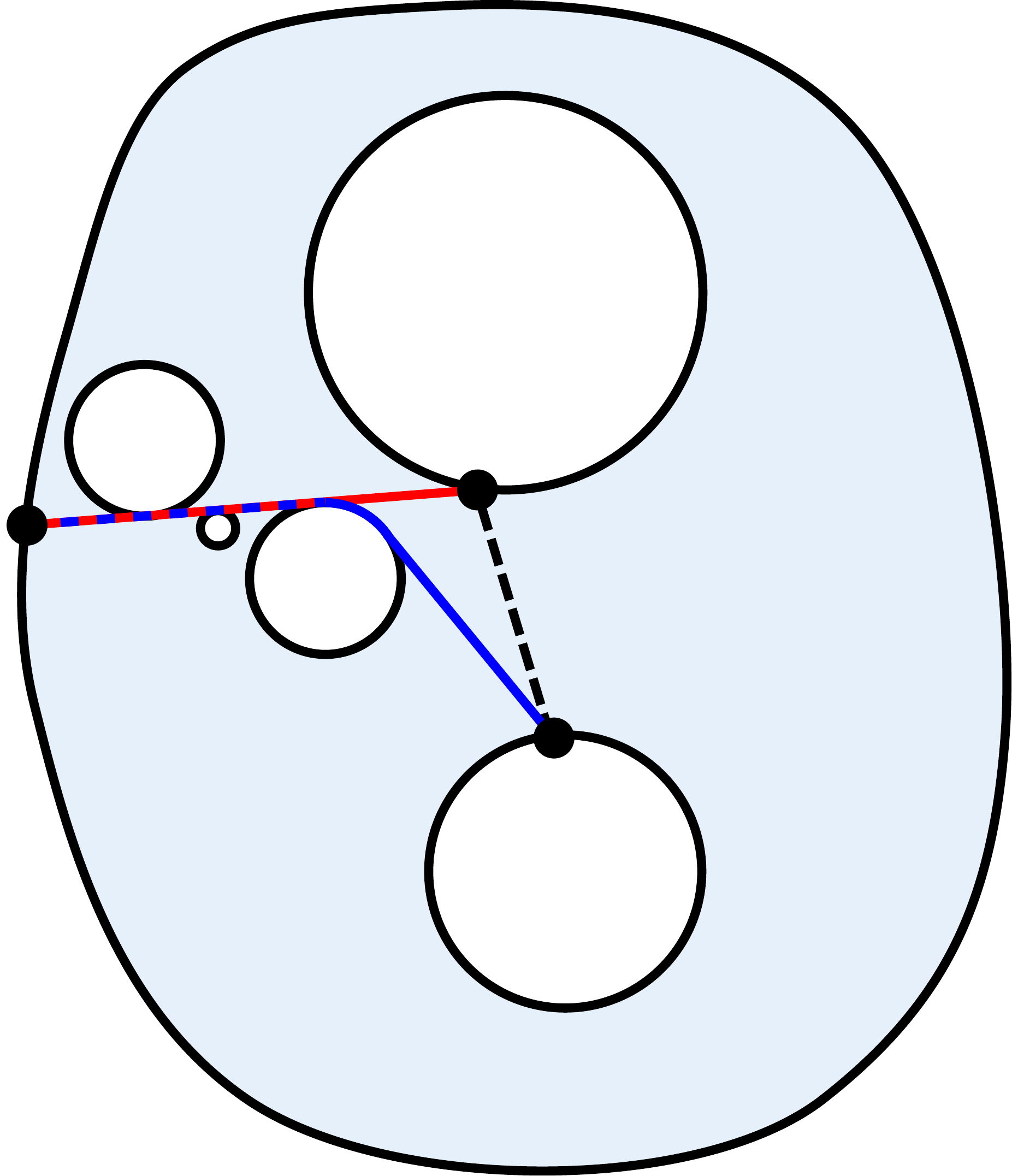}
    \caption{$q_3$}
    \label{fig:q3}
  \end{figure}
  $q_2$ will have a neighborhood like above if the overlapping part of $q_3$  is not tangent to the boundary in the middle.
  Assume the overlapping part of $q_3$  is tangent to the boundary in the middle.
  Call the overlapping part $l$.
  Here $l : [0, 1] \rightarrow M$ is a geodesic starting at the root of $q_2$.
  Let $X : [0, 1] \rightarrow \Omega M$ be a unit normal vector field along $l$ which is pointing inwards (as a unit normal vector of $\partial M$) at the root of $q_2$.
  Pick the smallest $a \in (0, 1]$, if it exists, such that
  $l$ is tangent to $\partial M$ at $l(a)$
  and that $X(a)$ is pointing outwards.
  If there is no such $a$,
  put $a = 1$.
  We obtain a pair $q_4$ by adding $l|_{[0, a]}$ to pairs in $q_2$ (Figure \ref{fig:q4}).
  \begin{figure}[h]
    \centering
    \includegraphics[width=0.2\linewidth]{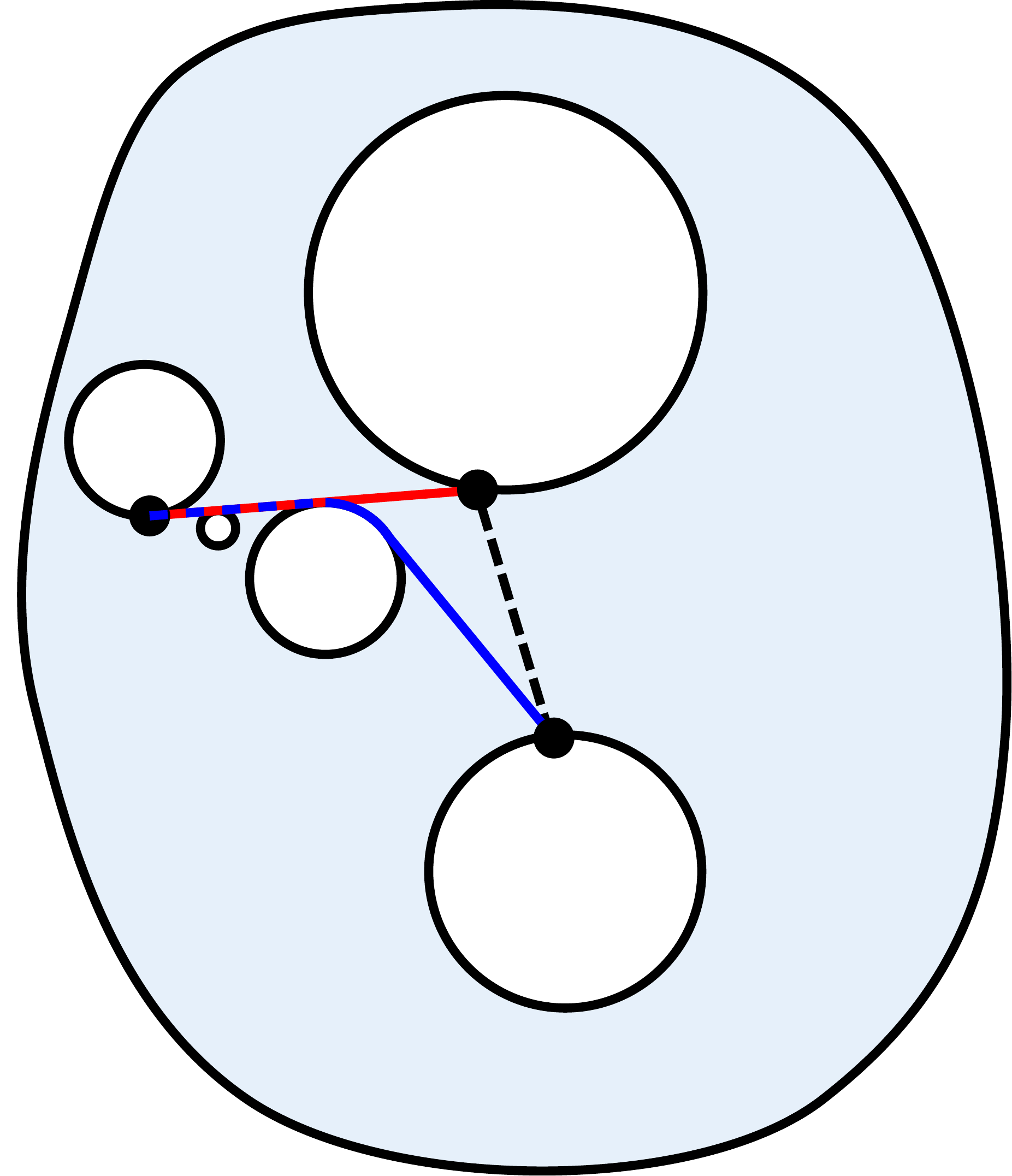}
    \caption{$q_4$}
    \label{fig:q4}
  \end{figure}
  We may move the root of $q_4$ along the boundary in two directions,
  obtaining pairs in $P_M(p, q)$.
  If we move the root along the boundary in the direction of $-l'(a)$,
  we will get pairs in $Q_M(p, q)$.
  If we move the root along the boundary in the direction of $l'(a)$,
  the overlapping part will persist.
  Notice that small neighborhoods of any other extensions of $q_2$ do not contain elements in $Q_M(p, q)$ because all their elements are pairs of geodesics with overlapping part at the ends.
  Hence $q_2$ has a neighborhood homeomorphic to $(0, 1)$,
  where one half of the neighborhood comes from half of a small neighborhood of $q_2$ in $P_M(p, q)$,
  and another half of the neighborhood comes from half of a small neighborhood of $q_4$ in $P_M(p, q)$.
  Note that $b(q_4) = q_2$.

  Therefore, $Q_M(p, q)$ is a 1-manifold without boundary.

\end{proof}

Suppose that there is a geodesic $\gamma$ from $p$ to $q$.
Let $\gamma_q$ be the constant curve at $q$,
$\gamma_p$ be the constant curve at $p$,
and $-\gamma$ be $\gamma$ with the opposite direction.
\begin{prop}
  $(\gamma, \gamma_q)$ and $(\gamma_p, -\gamma)$ are in the
  same component of $Q_M(p, q)$.
  Moreover,
  that component is homeomorphic to a circle.
  \label{prop:reverse}
\end{prop}

\begin{rem}
  This proposition also holds if $\gamma$ is a p-geodesic.
\end{rem}

\begin{proof}
  Without loss of generality, we assume that $M$ is simply connected but not necessarily compact.
  The curve $\gamma$ separate $M$ into two components.
  Namely,
  $M = D_1 \bigcup D_2$
  where $D_1 \bigcap D_2 = \gamma$.
  Without loss of generality,
  assume that $D_1$ is on the left of $\gamma$.
  Let $Q_1 = \{(\gamma_1, \gamma_2) \in Q_M : \gamma_1(1) \in D_1\}$,
  and $Q_2 = \{(\gamma_1, \gamma_2) \in Q_M : \gamma_1(1) \in D_2\}$.
  Let $Q_1^0$ be the component of $Q_1$
  which contains $(\gamma, \gamma_q)$.
  If $(\gamma_p, -\gamma) \notin Q_1^0$,
  then $Q_1^0$ has only one end point $(\gamma, \gamma_q)$.
  It follows that $Q_1^0$ is homeomorphic to $[0, \infty)$.
  Let $(\alpha, \beta) : [0, \infty) \rightarrow Q_1^0$
  be a homeomorphism.
  Here we reparametrize $\alpha(t)$ and $\beta(t)$ using the length parameter.
  We shall show that $\alpha(t)$ converges to a infinitely long p-geodesic ray $\alpha_0$
  as $t \rightarrow \infty$.

  Pick any $s \in [0, \infty)$.
  Notice that the space of p-geodesics starting at $\gamma_1(0)$
  whose length is less or equal to $s$ is homeomorphic to the closed ball
  $\overline{B_s(\tilde{p})}$ where $\tilde{p}$ is a lift of $\gamma_1(0)$ in $\tilde{M}$,
  which is compact.
  Hence there is $T_s \in [0, \infty)$ such that
  $\ell(\alpha(t)) > s$ when $t > T_s$.
  When $t > T_s$,
  $\alpha(t)|_{[0, s]}$ can either wrap around a concave boundary part in a fixed direction or stay still as $t$ increases.
  Notice that $\alpha(t)|_{[0, s]}$ can not wrap around a concave boundary part forever since it only has finite length $s$.
  Hence $\alpha(t)|_{[0, s]}$ converges pointwisely to a p-geodesic as $t$ converges to $\infty$.
  Therefore, $\alpha(t)$ converges to a infinitely long p-geodesic $\alpha_0$ pointwisely
  as $t \rightarrow \infty$.
  $\beta(t)$ converges to some $\beta_0$ similarly.

  Since $\alpha(t)$ and $\beta(t)$ only intersect at the root,
  $\alpha_0$ and $\beta_0$ do not intersect transversely.
  Use the length parameter on $\alpha_0$ such that $\alpha_0(0) = p$
  and that $\alpha_0(s)$ is defined for $s \ge 0$.

  Consider the p-geodesic ``triangle'' bounded by
  $\alpha_0$, $\beta_0$ and $\gamma$.
  For any $s > 0$,
  there is a maximal geodesic $g_s$ in the triangle which is tangent to $\alpha_0$ at $\alpha_0(s)$.
  Since $M$ has no trapped geodesics,
  $g_s$ has two end points.
  If one of the end points is on $\alpha_0$,
  then there will be two different p-geodesics between that end point and $\alpha_0(s)$,
  which contradicts Proposition \ref{prop:no_conj}.
  If both  of these two end points are on $\beta_0$,
  then there will be two different p-geodesics between these two end points,
  which contradicts Proposition \ref{prop:no_conj} again.
  Therefore, at least one of the end points are on $\gamma$.
  Let $X_s$ be the unit tangent vector based at that end point which is tangent to $g_s$ and pointing towards the inside of the triangle.
  There is a sequence $s_k \rightarrow \infty$ such that
  $X_{s_k}$ converges.
  By triangle inequality,
  $\ell(g_s) \ge s - \ell(\gamma)$.
  Hence $g_{s_k}$ converges to a geodesic with infinite length,
  which contradicts our assumption that $M$ has no trapped geodesics.

  Therefore, $Q_1$ is homeomorphic to $(0, 1)$.
  Similarly, $Q_2$ is also homeomorphic to $(0, 1)$.
  Thus
  the component of
  $Q_M(p, q)$
  which contains
  $(\gamma, \gamma_q)$ and $(\gamma_p, -\gamma)$
  is a circle.
\end{proof}

\section{convex part of the boundary}
The goal of this section is to prove the following proposition.
\begin{prop}
  For any $p_0 \in S$,
  the curvature of $\partial M$ at $p_0$ is the same as the curvature of $\partial N$ at $h(p_0)$.
  \label{prop:no_trapped}
\end{prop}

Pick any $p_0 \in S_+$. We aim to show that $\partial N$ is also convex at $h(p_0)$.
Pick a unit tangent vector $X_0 \in \partial_0 \Omega_{p_0} M$.
There are two choices of $X_0$ but either works.

For any $\theta \in (0, \pi)$, let $X_\theta$ be the unit tangent vector in $\partial_+ \Omega_{p_0} X$ such that the angle between $X_0$ and $X_\theta$ is $\theta$.
Since $S$ is convex at $p_0$,
there is $\delta > 0$ such that
$\gamma_{X_\theta}$ is not tangent to $\partial M$ when $\theta \in (0, \delta)$.
Since $M$ and $N$ have the same scattering data,
$e(\gamma_{X_\theta}) = \ell(\gamma_{\Phi(X_\theta)}) - \ell(\gamma_{X_\theta})$ is equal to a fixed constant $L > 0$
when $\theta \in (0, \delta)$.
If $L = 0$, then $N$ is also convex at $h(p_0)$ \cite{Mi}.

Assume that $L \neq 0$, then there is a closed geodesic in $N$ of length $L$ which is tangent to $\partial N$ at $h(p_0)$.
We shall show that this is impossible in this section.

Let $S_1$ be the component of $\partial M \setminus S_-$  which contains $p_0$.
$S_1$ is either a closed circle,
or a curve with two ends.

\subsection{$S_1$ is a curve with two ends}

For any $Y \in \Omega \partial M$,
let $\overline\gamma_Y$ be the maximum geodesic ray whose initial tangent vector is $Y$.
If $\overline\gamma_Y$ is just a point, or if it only runs along a totally geodesic part of $\partial M$ and never leaves the boundary, then $Y$ is called a convex direction.
Otherwise, $Y$ is called a concave direction.

$\partial M$ is orientable since it is one-dimensional.
Fix an orientation on $\partial M$.
Then at each $p \in \partial M$,
there is a positively oriented unit tangent vector $Y_+(p) \in \Omega_p \partial M$,
and a negatively oriented unit tangent vector $Y_-(p) \in \Omega_p \partial M$.

Let $Y_n$ be a sequence of a unit tangent vectors based at $p$, pointing inwards and converging to $Y_+(p)$.
Then define
\begin{align*}
  L_+(p) :=
  \begin{cases}
    \lim_{n \rightarrow \infty} e(\gamma_{Y_n}) & \text{if $Y_+(p)$ is a convex direction,}\\
    0 & \text{if $Y_+(p)$ is a concave direction.}
  \end{cases}
\end{align*}
$L_-(p)$ is defined similarly.
\begin{prop}
  $L_+ = L_-$ and they are constant on each component of $\partial M$.
\end{prop}
\begin{proof}
  $L_+(p) = L_-(p) = 0$ when both $Y_+(p)$ and $Y_-(p)$ are concave directions.
  It is also obvious that $L_+(p) = L_-(p)$ when both $Y_+(p)$ and $Y_-(p)$ are convex directions.
  (Both are equal to the length of a closed geodesic tangent to the boundary.)
  Also, by the first variation formula, $L_\pm$ are constant near $p$ if $p$ is not a switch point.

  Suppose that $p_1$ is a switch point such that $L_+$ is non-constant near $p_1$.
  Without loss of generality, assume that $Y_+(p_1)$ is a convex direction
  Pick a small open neighborhood of $p_1$ in $\partial M$.
  The neighborhood is separated into two parts by $p_1$
  Pick a point $p_2$ from the part which $Y_+(p)$ points to,
  and a point $p_3$ from the other part.
  Since there are only finite many switch points,
  when $p_2$ and $p_3$ are close enough,
  any geodesic going through $p_2$ or $p_3$, except possibly a geodesic going though $p_3$ and $p_1$, are not tangent $\partial M$ at switch points.

  Pick $q \in \partial M \setminus F_M$ such that there is a geodesic $\gamma_1$ which intersects $\partial M$ transversely at $p_1$ and $q$.
  When $p_2$ and $p_3$ are close enough,
  there is also a geodesic $\gamma_2$ which intersects $\partial M$ transversely at $p_2$ and $q$,
  a geodesic $\gamma_3$ which intersects $\partial M$ transversely at $p_3$ and $q$,
  and a continuous family of geodesics between them which intersect $\partial M$ transversely.
  Then we have $e(\gamma_2) = e(\gamma_3)$ by Proposition \ref{prop:constante}.

    \begin{figure}[h]
      \labellist
      \small\hair 2pt
      \pinlabel $p_2$ [t] at 116 99
      \pinlabel $p_3$ [t] at 155 110
      \pinlabel $\gamma_6$ [b] at 139 108
      \pinlabel $q$ [bl] at 137 191
      \pinlabel $\gamma_2$ [r] at 126 146
      \pinlabel $\gamma_3$ [l] at 147 150
      \endlabellist
      \centering
      \includegraphics[width=0.4\linewidth]{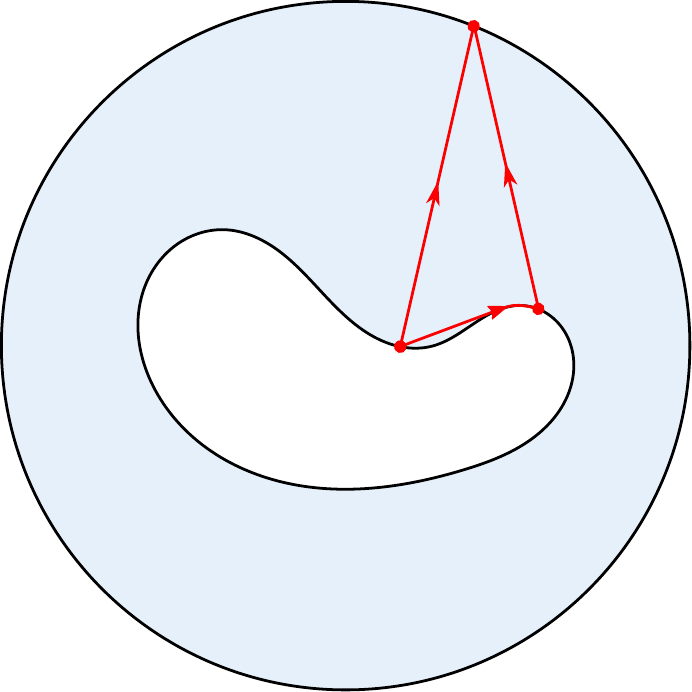}
      \caption{$H_1(\theta_0, \cdot) = H_2(\theta_0, \cdot)$}
      \label{fig:constantL}
    \end{figure}

  For any p-geodesic $\gamma \in \Gamma^p_M$ starting at $p_2$,
  by Proposition \ref{prop:no_conj},
  there is a unique p-geodesic $H(\gamma)$ close to $\gamma$ such that $H(\gamma)(0) = p_3$ and that $H(\gamma)(1) = \gamma(1)$.

  Recall that $l$ is constant on each component of $Q_M^0(p_2, q)$.
  Pick any $(\gamma_4, \gamma_5) \in Q_M(p_2, q) \setminus Q_M^0(p_2, q)$
  where both $\gamma_4$ and $\gamma_5$ are non-constant.

  We claim that $(H(\gamma_4), \gamma_5) \in Q_M(p_3, q) \setminus Q_M^0(p_3, q)$.
  The root of $(\gamma_4, \gamma_5)$ is a switch point
  and either $\gamma_4$ or $\gamma_5$ is tangent to $\partial M$ at that point.
  If $\gamma_5$ is tangent to $\partial M$ at the that switch point,
  then we have
  $(H(\gamma_4), \gamma_5) \in Q_M(p_3, q) \setminus Q_M^0(p_3, q)$.
  If $\gamma_4$ is tangent to $\partial M$ at that switch point,
  then $\gamma_4$ is not a geodesic,
  since we assume that geodesics going through $p_2$ are not tangent to $\partial M$ at switch points.
  Hence the ending part of $\gamma_4$ is a geodesic tangent to $\partial M$ at both ends.
  Therefore, when $p_2$ and $p_3$ are close enough,
  the ending part of $\gamma_4$ and $H(\gamma_4)$ will coincide,
  which implies that $H(\gamma_4)$ is also tangent to $\partial M$ at $\gamma_4(1)$.
  Thus, $(H(\gamma_4), \gamma_5) \in Q_M(p_3, q) \setminus Q_M^0(p_3, q)$.

  Now, consider the two components $I_2$ and $J_2$ of $Q_M(p_2, q) \setminus Q_M^0(p_2, q)$
  which are adjacent to $(\gamma_4, \gamma_5)$, and corresponding components $I_3$ and $J_3$ of
$Q_M(p_3, q) \setminus Q_M^0(p_3, q)$ which are adjacent to $(H(\gamma_4), \gamma_5)$.
  If $\gamma_4$ is tangent to $\partial M$ at the root,
  then the ending part of $\gamma_4$ and $H(\gamma_4)$ coincide,
  which implies that
  \begin{align}
    l(J_2) - l(I_2) = l(J_3) - l(I_3).
    \label{eq:same_diff}
  \end{align}
  If $\gamma_4$ is transverse to $\partial M$ at the root,
  then $H(\gamma_4)$ is also transverse to $\partial M$ if $p_2$ and $p_3$ are close enough,
  so we also have \eqref{eq:same_diff}.

  Let $\gamma_6$ be the shortest p-geodesic from $p_2$ to $p_3$.
  By proposition \ref{prop:reverse},
  $(\gamma_2, \gamma_q)$ and $(\gamma_{p_2}, -\gamma_2)$ are in the
  same component of $Q_M(p_2, q)$ which is homeomorphic to a circle.
  Call that component $Q_2$.
  Since $(\gamma_6, -\gamma_3)$ and $(\gamma_{p_2}, -\gamma_2)$ are close to each other,
  $(\gamma_6, -\gamma_3) \in Q_2$.
  Since $Q_2$ is a circle,
  $(\gamma_6, -\gamma_3)$ and $(\gamma_2, \gamma_q)$ separate $Q_2$ into two parts.
  Denote the part which does not contain $(\gamma_{p_2}, -\gamma_2)$ by $U_2$.
  Here $U_2$ does not contain end points.

  By proposition \ref{prop:reverse},
  $(\gamma_3, \gamma_q)$ and $(\gamma_{p_3}, -\gamma_3)$ separate the component of $Q_M(p_3, q)$ which contains them into two parts.
  Denote the part which does not contain $(-\gamma_6, -\gamma_2)$ by $U_3$.
  Here $U_3$ does not contain end points.

  Let $V_2$ be the component of $U_2 \setminus Q_M^0(p_2, q)$
  which is adjacent to $(\gamma_2, \gamma_q)$,
  and $V_3$ be the component of $U_3 \setminus Q_M^0(p_3, q)$
  which is adjacent to $(\gamma_3, \gamma_q)$.
  Then $l(V_2) = e(\gamma_2) - L_\pm(q) = e(\gamma_3) - L_\pm(q) = l(V_3)$.
  Note that choice between $L_+(q)$ and $L_-(q)$ depends on the orientation of the boundary,
  but it does not affect the computation here.
  Let $V_2'$ be the component of $U_2 \setminus Q_M^0(p_2, q)$
  which contains $(\gamma_6, -\gamma_3)$,
  and $V_3'$ be the component of $U_3 \setminus Q_M^0(p_3, q)$
  which is adjacent to $(\gamma_{p_3}, -\gamma_3)$.
  By \eqref{eq:same_diff},
  $l(V_2') = l(V_2) = l(V_3) =  l(V_3')$.
  Notice that $l(V_3) = L_+(p_3) - e(\gamma_3)$
  but $l(V_3') = e(\gamma_6) - e(\gamma_3) = L_+(p_2) - e(\gamma_3)$.
  Hence $L_+(p_2) = L_+(p_3)$.
  Hence $L_+$ is constant on each component of $\partial M$.
  Similarly, $L_-$ is also constant on each componnet of $\partial M$.
  Since $M$ has no trapped geodesics,
  there is always a point $p$ on each component of $\partial M$ such that $\partial M$ is either strictly convex near $p$ or strictly concave near $p$,
  which implies that
  $L_+(p) = L_-(p)$.
  Hence $L_+ = L_-$ and they are constant on each component of $\partial M$.
\end{proof}

From now, we shall replace both $L_+$ and $L_-$ by $L$.
When $S_1$ is a curve with two ends,
then the component of $\partial M$ which contains $S_1$ must contain a concave part,
and thus $L = 0$ on $S_1$.
Hence Proposition \ref{prop:no_trapped} holds.

\subsection{$S_1$ is closed}

Now we assume that $S_1$ is closed.

If every geodesic ray starting at $p_0$ intersects $\partial M$ only on $S_1$,
then the old argument for simple manifolds works automatically.
Assume that there is a geodesic ray starting at $p_0$ intersects $\partial M$ at a point on $\partial M \setminus S_1$.
Assume that $L > 0$.

  Fix an orientation on $\partial N$ and let $X_0(x)$
  be the unit vector tangent to $\partial N$ at $x \in \partial N$
  such that $X_0(x)$ and $\partial N$ have the same orientation.
  Let $h_1 : \real / \integer \rightarrow h(S_1)$
  be an orientation preserving diffeomorphism.

  For each $x \in \partial N$ and $\theta \in (0, \pi)$, let $X_\theta(x) \in \partial_+ \Omega N$ be the unit tangent vector
  such that the angle between $X_0(x)$ and $X_\theta(x)$ is $\theta$.
  For each $\theta \in (-\pi, 0)$, put $X_\theta(x) = - X_{\pi - \theta}(x)$.
  Pick a small $\delta_1 > 0$ such that $\gamma_{X_\theta(x)}$ intersects $\partial N$ transversely
  for any $x \in h(S_1)$ and $\theta \in (0, \delta_1)$.
  Pick a small $\delta \in (0, \delta_1)$
  Write $Y_1(x) = X_\delta(x)$.
  Define $T : h(S_1) \rightarrow h(S_1)$ as $T(x) := \pi(\alpha_N(Y_1(x)))$.
  Here we choose $\delta$ small enough such that the two end points of $\gamma_{X_\delta(x)}$
  move in the same direction as $x$ moves on $h(S_1)$.
  In other words, $T$ is a homeomorphism.

  Let $Y_2(x) = \alpha_N(Y_1(T^{-1}(x)))$.
  Then $Y_1(x)$ and $Y_2(x)$ separate the circle $\Omega_x N$
  into two segments.
  Let $A(x)$ be the segment containing $X_0(x)$ (which is the shorter segment).
  Then $A = \bigcup_{x \in \partial N} A(x)$ is an annulus with boundaries
  $Y_1(\partial N)$ and $Y_2(\partial N)$.
  Thus we can can break $A$ down to a family of \emph{disjoint} curves
  $\eta_x : [0, 1] \rightarrow A$
  from $Y_1(x)$ to $Y_2(T(x)) = \alpha_N(Y_1(x))$.
  See Figure \ref{fig_eta}.
  \begin{figure}[h!]
      \center
      \includegraphics[width=0.4 \textwidth]{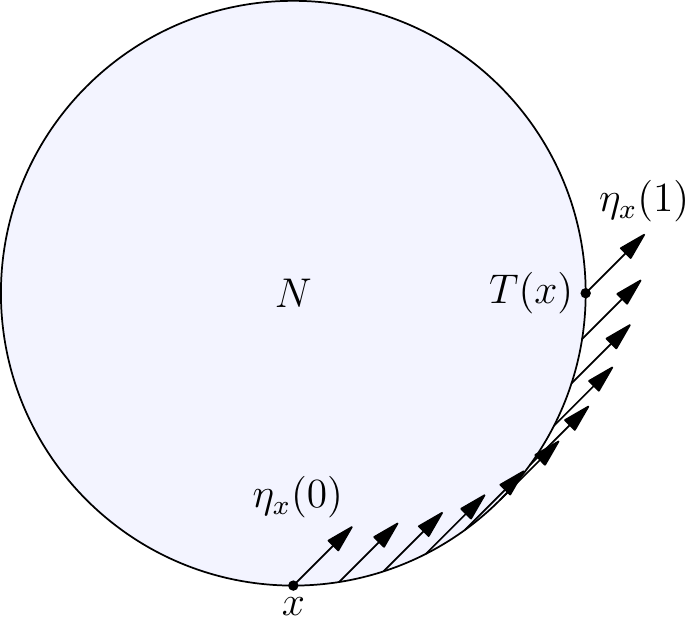}
      \caption{Values of $\eta_x$ from $0$ to $1$}
      \label{fig_eta}
  \end{figure}
  \begin{defn}
    The \emph{unit tangent vector field}  of
    a smoothly immersed curve $\gamma$ on any
    Riemannian surface $N^2$ (possibly with boundary) is a smoothly immersed curve
    $\tilde{\gamma}$ in $\Omega N$ defined as
    \begin{align}
      \tilde{\gamma}(t) = \left(\gamma(t), \frac{\gamma'(t)}{|\gamma'(t)|}\right).
      \label{eq:unit_tangent}
    \end{align}
  \end{defn}

  \begin{defn}
    Let $P : \Omega N \rightarrow P \Omega N$ be the quotient map
    on the unit tangent bundle which identifies the opposite vectors based
    at the same point.
    For any smoothly immersed curve
    $\gamma$ in $N^2$,
    $P \circ \tilde{\gamma}$ is called
    the \emph{projectivized unit tangent vector field} (or the \emph{tangent line field}) of $\gamma$.
  \end{defn}

  \begin{figure}[h!]
      \center
      \includegraphics[width=0.4 \textwidth]{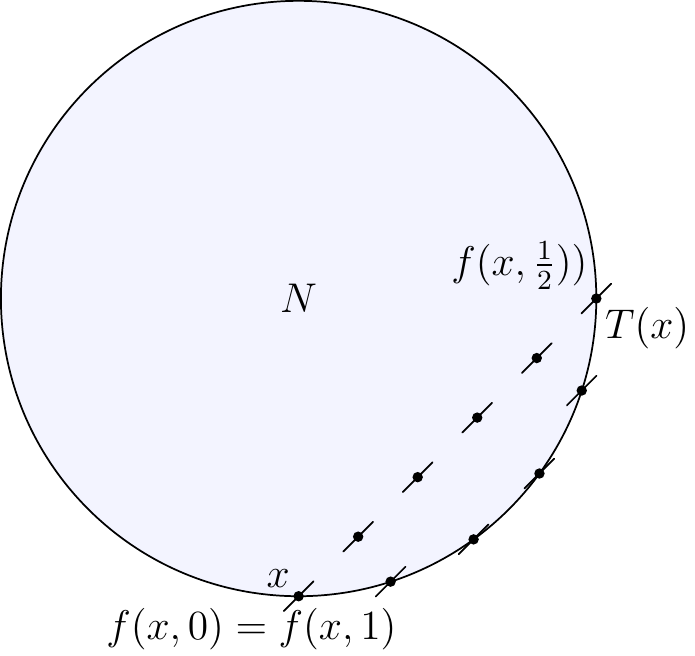}
      \caption{Values of $f(x, \cdot)$ from $0$ to $1$. This might be deceiving since $\gamma_{Y_1(x)}(2t)$ should have self-intersections.}
      \label{fig_f}
  \end{figure}
  Define
  \begin{align*}
      f : \partial N \times \real / \integer \rightarrow P \Omega N
  \end{align*}
  as
  \begin{align*}
      f(x, t) =
      \begin{cases}
	P\left(\tilde\gamma_{Y_1(x)}(2 t)\right) & \text{if $0 \le t \le \frac{1}{2}$,}\\
	P\left(\eta_x(2 - 2t)\right) & \text{if $\frac{1}{2} \le t \le 1$.}
      \end{cases}
  \end{align*}
  See Figure \ref{fig_f}.
  \begin{prop}
    $f : \partial N \times \sphere \rightarrow P \Omega N$ is an embedding.
  \end{prop}
  \begin{proof}
    The proof is the same as the proof of \cite[Proposition 4.2]{We}.
  \end{proof}

  \begin{prop}
    $f(x, \cdot)$ is contractible in $P \Omega N$.
  \end{prop}

  \begin{proof}
    Put $x_0 = h(p_0)$.
    Let $\theta_0$ be the smallest positive number such that
    $\gamma_{\Phi^{-1}(X_{\theta_0})}$ is tangent to $\partial M$ at a point $p_1 \in \partial M \setminus S_1 $.
    For any $\theta \in (0, \theta_0)$,
    $\gamma_{\Phi^{-1}(X_{\theta})}$ intersects $\partial M$ transversely.
    Hence $\gamma_{\Phi^{-1}(X_{\theta})}$ is a continuously varying family of disjoint curves,
    which implies that their union is a simply connected convex region.
    Thus each geodesic ray starting at $p_1$ intersects $S_1$ transversely because there are no conjugate points.
    For each $\theta \in (0, \theta_0)$,
    let $\gamma_\theta : [0, 1] \rightarrow N$ be the
    geodesic from $\pi(\alpha_N(X_\theta(p_0)))$ to $h(p_1)$.
    Also define $\gamma_0$ and $\gamma_{\theta_0}$ by taking limits.
    Let $a(\theta)$ be the angle between $\gamma_{\theta_0}$ and $\gamma_\theta$.

    We will construct two homotopies
    $H_1 : [\delta, \theta_0] \times \real / \integer \rightarrow \Omega N$
    and
    $H_2 : [0, \theta_0] \times \real / \integer \rightarrow \Omega N$.
    $P \circ H_1(\delta, \cdot)$ will be homotopic to $f(x_0, \cdot)$.
    $H_1(\theta_0, \cdot)$ and $H_2(\theta_0, \cdot)$ will be the same.
    Then we will show that $H_2(0, \cdot)$ is contractible,
    which will imply that $f(x_0, \cdot)$ is contractible.

    \begin{figure}[h]
      \centering
      \includegraphics[width=0.4\linewidth]{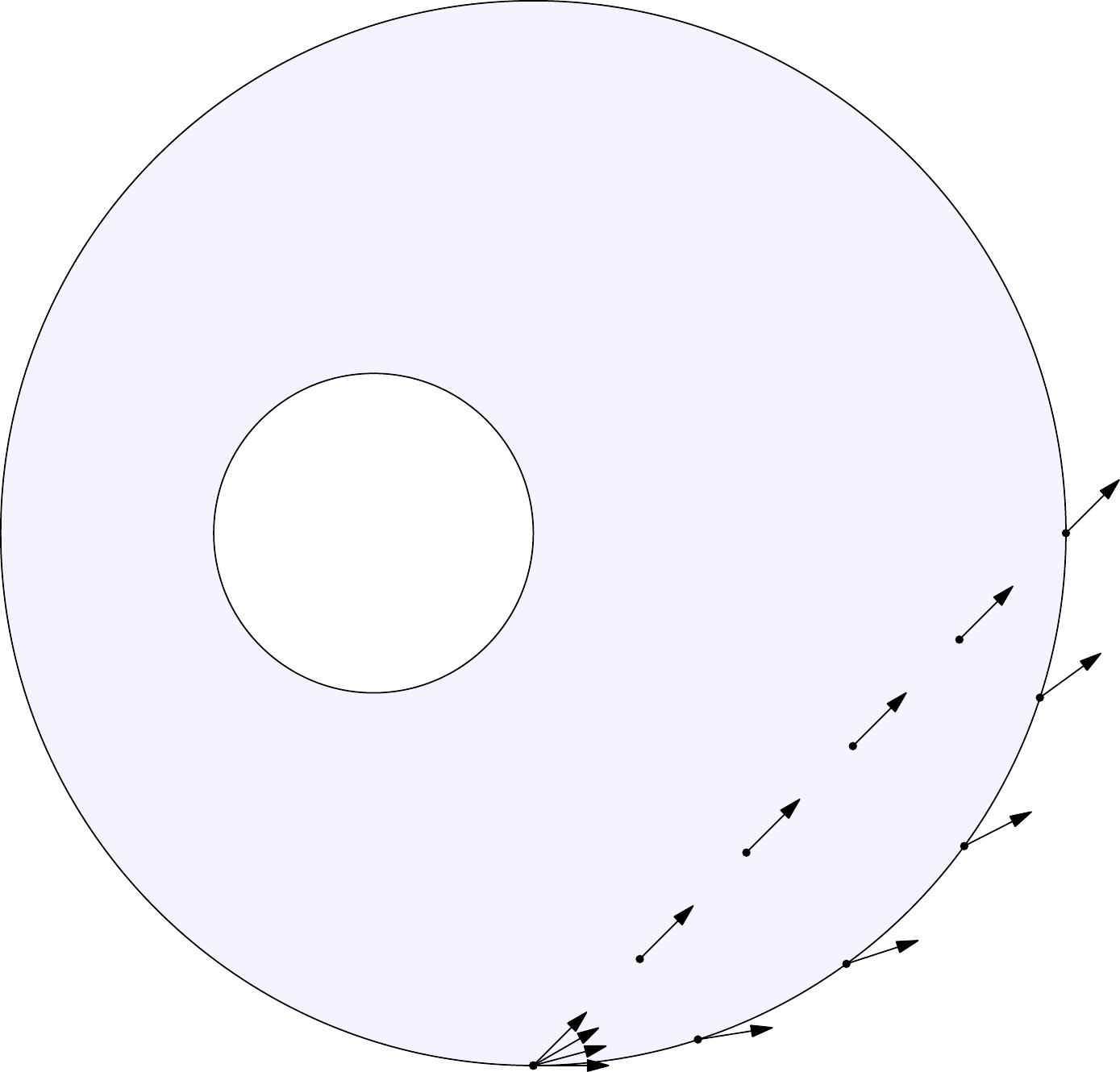}
      \caption{This is $H_1(\delta, \cdot)$. $P \circ H_1(\delta, \cdot)$ is homotopic to $f(x_0, \cdot)$}
      \label{fig:81}
    \end{figure}
    \begin{figure}[h]
      \labellist
      \small\hair 2pt
      \pinlabel $h(p_1)$ [l] at 186 188
      \endlabellist
      \centering
      \includegraphics[width=0.4\linewidth]{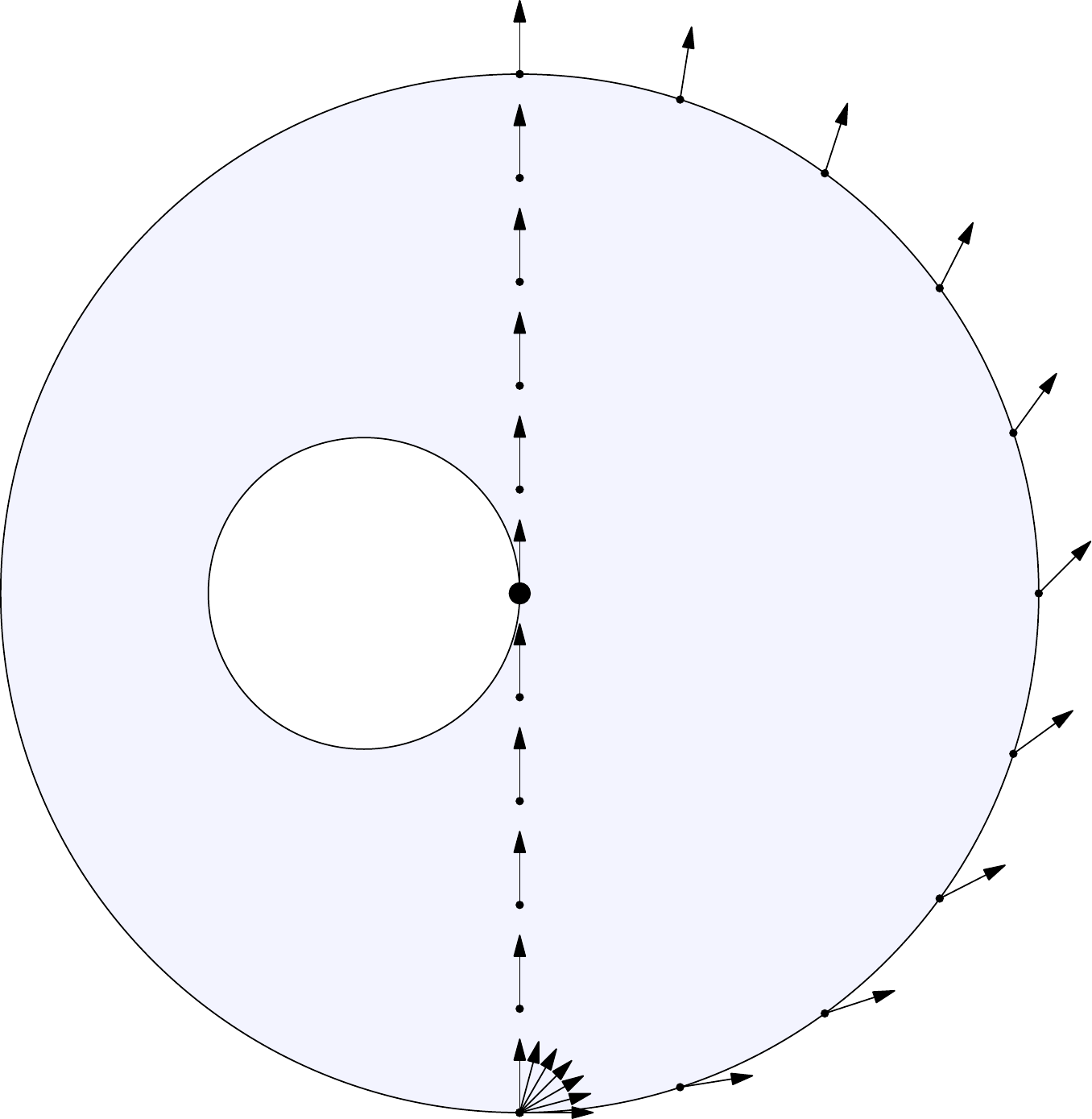}
      \caption{$H_1(\theta_0, \cdot) = H_2(\theta_0, \cdot)$}
      \label{fig:81}
    \end{figure}
    \begin{figure}[h]
      \centering
      \includegraphics[width=0.4\linewidth]{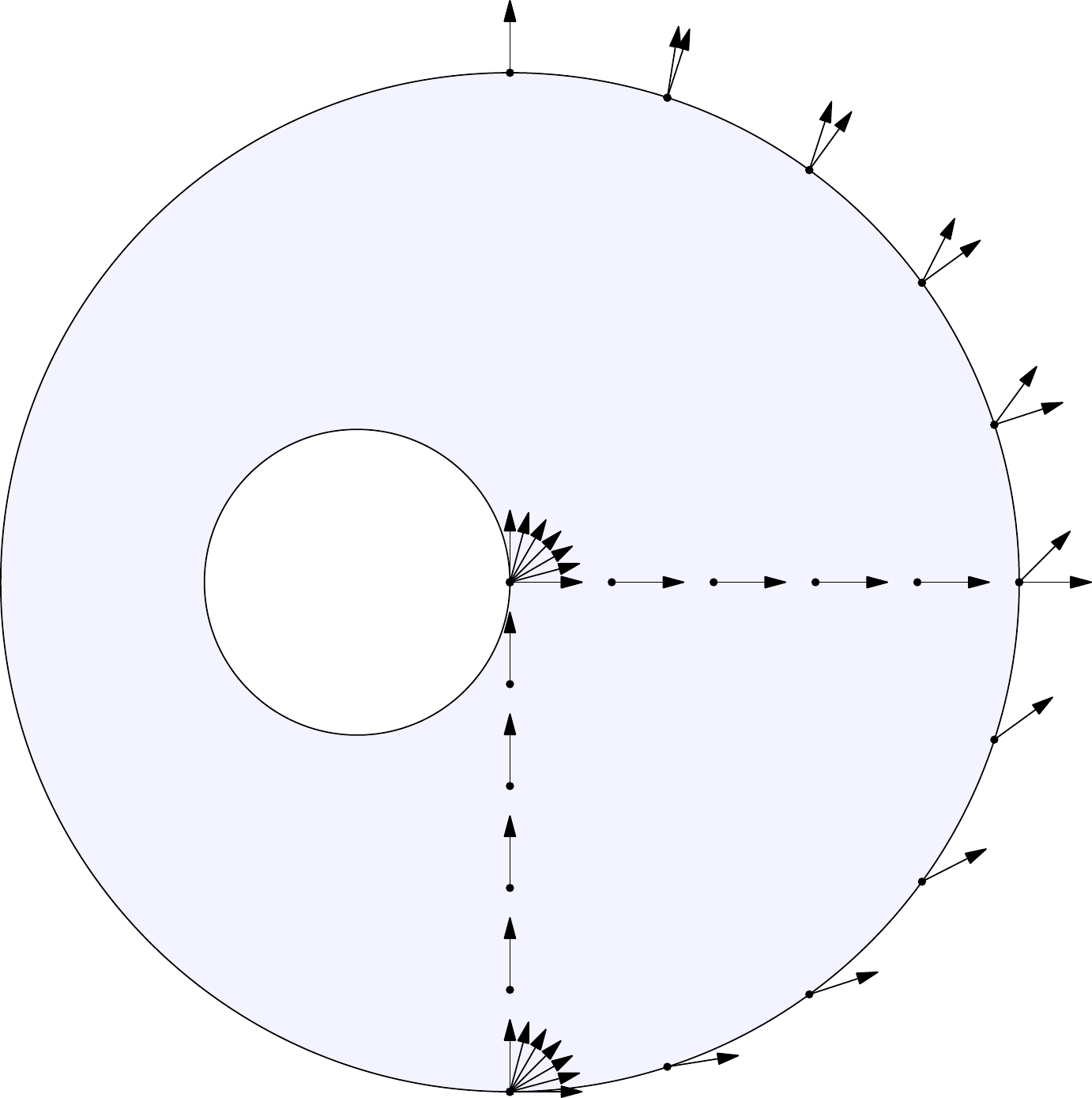}
      \caption{$H_2(\frac{\theta_0}{2}, \cdot)$}
      \label{fig:81}
    \end{figure}
    \begin{figure}[h]
      \centering
      \includegraphics[width=0.4\linewidth]{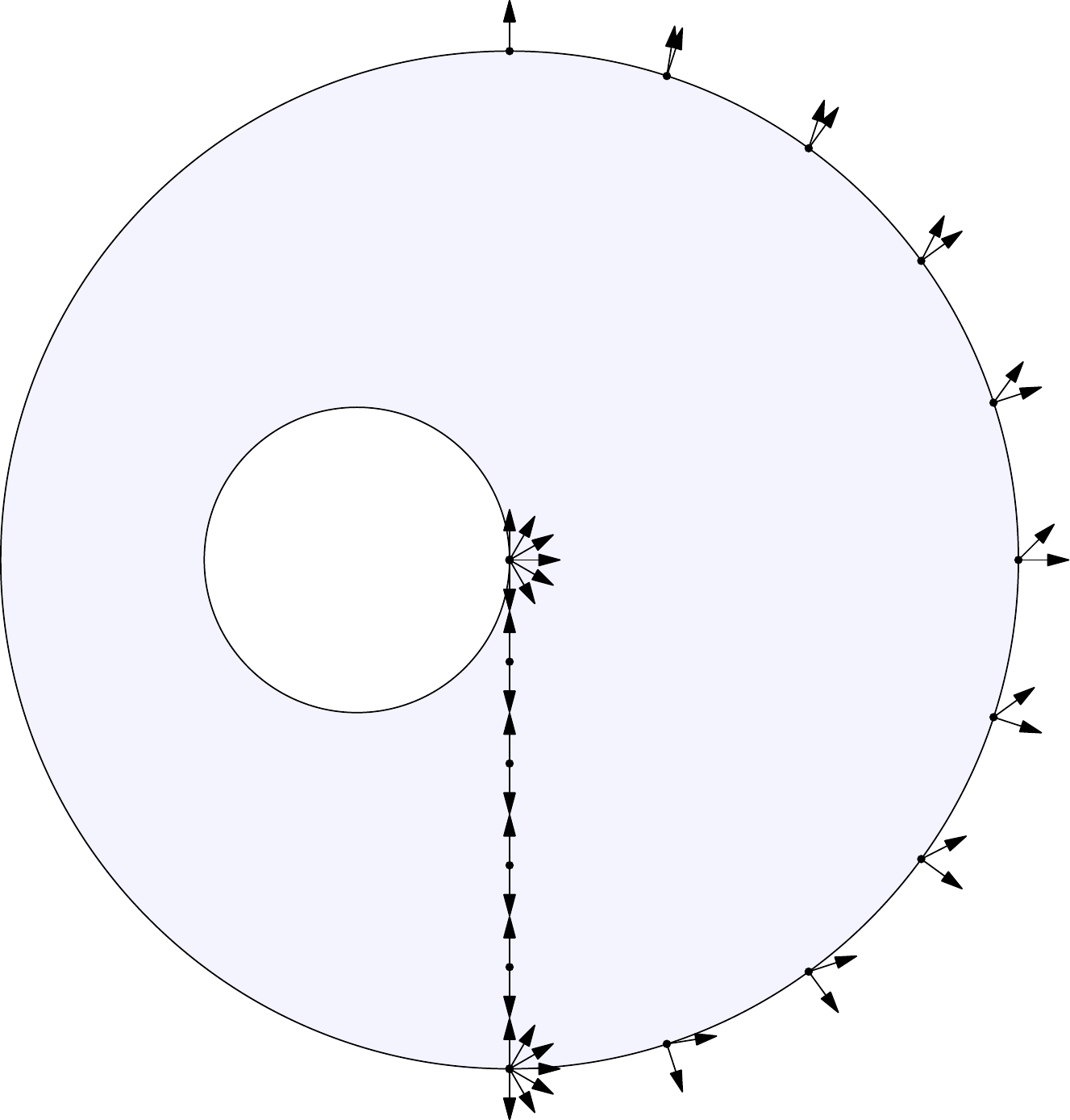}
      \caption{$H_2(0, \cdot)$ is contractible in this graph since no vector points to the left.}
      \label{fig:81}
    \end{figure}

    Consider the homotopy
    $H_1 : [\delta, \theta_0] \times \real / \integer \rightarrow \Omega N$ defined as
    \begin{align*}
      H_1(s, t) =
      \begin{cases}
	\tilde\gamma_{X_s}(3 t) & \text{if $0 \le t \le \frac{1}{3}$,}\\
	\alpha_N( X_{(2 - 3t)s}(x_0)) & \text{if $\frac{1}{3} \le t \le \frac{2}{3}$,}\\
	X_{(3t -2)s}(x_0) & \text{if $\frac{2}{3} \le t \le 1$.}\\
      \end{cases}
    \end{align*}
    Here $P \circ H_1(\delta, \cdot)$ is homotopic to $f(x_0, \cdot)$.

    Define $H_2 : [0, \theta_0] \times \real / \integer \rightarrow \Omega N$ as
    \begin{align*}
      H_2(s, t) =
      \begin{cases}
	\tilde\gamma_0(5 t) & \text{if $0 \le t \le \frac{1}{5}$,}\\
	-\tilde\gamma_{\theta_0 - (6t - 1)(\theta_0 - s)}(1)& \text{if $\frac{1}{6} \le t \le \frac{2}{6}$,}\\
	-\tilde\gamma_{s}(3 - 6 t) & \text{if $\frac{2}{6} \le t \le \frac{3}{6}$,}\\
	-\tilde\gamma_{\theta_0 - (4 - 6t)(\theta_0 - s)}(0) & \text{if $\frac{3}{6} \le t \le \frac{4}{6}$,}\\
	\alpha_N( X_{(5 - 6t)\theta_0}(x_0)) & \text{if $\frac{4}{6} \le t \le \frac{5}{6}$,}\\
	X_{(6t -5)\theta_0}(x_0) & \text{if $\frac{5}{6} \le t \le 1$.}\\
      \end{cases}
    \end{align*}
    Here $H_2(\theta_0, \cdot)$ is a reparametrization of $H_1(\theta_0, \cdot)$.

    Notice that $H_2(0, \cdot) |_{(0, \frac{1}{2})}$ lies entirely on $\gamma_0$.
    Hence we can homotope $H_2(0, \cdot) |_{[0, \frac{1}{2}]}$
    to a curve in $\Omega_{x_0} N$.
    More precisely,
    $H_2(0, \cdot) |_{[0, \frac{1}{2}]}$ is homotopic to
    $\beta : [0, \frac{1}{2}] \rightarrow \Omega_{x_0} N$ defined as
    $\beta(t) = X_{(1 - 4t)\theta_0}(x_0)$.
    Similarly,
    $H_2(0, \cdot) |_{(\frac{1}{2}, 1)}$ is homotopic to $-\beta$.
    Hence $H_2(0, \cdot)$ is contractible.
    Since $P \circ H_2(0, \cdot)$ is homotopic to $f(x_0, \cdot)$,
    $f(x_0, \cdot)$ is contractible in $P \Omega N$.
  \end{proof}

  We can prove the following proposition using the same proof of \cite[Proposition 4.4]{We}.
  \begin{prop}
    $f(x_0, \cdot)$ is isotopically trivial in $P \Omega N$.
    \label{prop:trivial'}
  \end{prop}

  However, by \cite[Theorem 3.14]{We},
  $f(x_0, \cdot)$ is isotopically non-trivial, which contradicts Proposition \ref{prop:trivial'}.
  Hence $L = 0$.
\section{Proof of the main Theorem}

Proposition \ref{prop:no_trapped} enable us extending $\Phi$ to $\Gamma^p_M$.
\begin{prop}
  There is a map $\Phi : \Gamma^p_M \rightarrow \Gamma^p_N$
   which satisfies the following conditions.
  \begin{enumerate}
    \item $\Phi$ is continuous with respect to the compact open topology.
    \item For any $\gamma(t) \in \partial M$, reparametrizing $\Phi(\gamma)$ if necessary, we have $\Phi(\gamma)(t) = h(\gamma(t))$
      and $\frac{\Phi(\gamma)'(t)}{|\Phi(\gamma)'(t)|} = \varphi(\frac{\gamma'(t)}{|\gamma'(t)|})$.
  \end{enumerate}
\end{prop}
\begin{proof}
  Recall that $\Gamma^1_M$ is the space of non-constant geodesics $[0, 1] \rightarrow M$ whose end points are on $\partial M$ and which are not tangent to $\partial M$.

  Let $\Gamma^4_M$ be the space of p-geodesics in $M$ which run along the non-convex part of $\partial M$.
  Note that all constant geodesics are in $\Gamma^4_M$.

  Recall that $\Phi : \Gamma^1_M \rightarrow \Gamma^p_N$ is defined
  as $\Phi(\gamma_X) = \gamma_{\varphi(X)}$.
  We extend $\Phi$ to a continuous map from $\overline{\Gamma^1_M}$ to $\Gamma^p_N$ by taking limits.
  This is well-defined because $M$ and $N$ have the same scattering data.
  For each $\gamma \in \Gamma^4_M$,
  we define $\Phi(\gamma) := h \circ \gamma$.

  Now, pick any $\gamma \in \partial \Gamma^p_M$.
  Since $F_M$ is finite, we have a decomposition
  $\gamma = \gamma_1 * \gamma_2 * \dots * \gamma_n$
  where each $\gamma_k$ is either in $\Gamma^2_M$ or in $\partial M$.
  Define $\Phi(\gamma) = \Phi(\gamma_1) * \Phi(\gamma_2) * \dots * \Phi(\gamma_n)$.
\end{proof}

Now, we may extend $e$ to $\Gamma^p$ and extend $l$ to $Q_M(p, q)$.
As before, $l$ is constant on each component of $Q_M(p, q)$.
Finally, the main theorem is a easy consequence of Proposition \ref{prop:reverse}.

\begin{proof}[Proof of Theorem \ref{thm:main}]
  Since $(\gamma, \gamma_q)$ and $(\gamma_p, -\gamma)$ are in the
  same component of $Q_M(p, q)$,
  $l(\gamma, \gamma_q) = l(\gamma_p, -\gamma)$.
  Hence $e(\gamma) - e(\gamma_q) = e(\gamma_p) - e(-\gamma)$.
  $e(\gamma_q) = \ell(\Phi(\gamma_q)) - \ell(\gamma(q)) = 0 - 0 = 0$.
  Similarly, $e(\gamma_p) = 0$.
  It is also obvious that $e(\gamma) = e(-\gamma)$.
  Hence $e(\gamma) = 0$,
  which implies that $M$ and $N$ have the same lens data.

\end{proof}

%
%
%

\section{Manifolds with boundary and no conjugate points}
\label{NCP}

In this section we consider compact Riemannian manifolds $M$ with smooth
boundary $\partial M$ such that no geodesic segment of $M$ has conjugate
points.  Our main goal is to prove Proposition \ref{minimizing}.

Standard applications of the second variation formula (see for example \cite%
{Do}) about a geodesic segment $\gamma$ having no conjugate points yields
that for every nontrivial differentiable one parameter family $\gamma_s(t)$
of curves in the space $\mathcal{C}_{(x,y)}$ of curves from $x=\gamma(0)$ to
$y=\gamma(L)$ with $\gamma_0=\gamma$ we have for the energy function $%
E(\gamma_s)=\int_0^L|\gamma_s^{\prime}(t)|^2dt$ that $\frac {d^2}{ds^2}%
|_{s=0}E(\gamma_s)>0$. (Here nontrivial means the variation field is not the
0 field.) The fact that $\frac {d}{ds}|_{s=0}E(\gamma_s)=0$ just follows
from the fact that $\gamma$ is a geodesic. This tells us (since we can
reduce to a finite dimensional space of piecewise geodesics) that $\gamma$
is a strict local minimum of the energy (and the length) in the space of
piecewise smooth paths between the endpoints.

\bigskip

The finiteness condition on $F_M$ eliminates the problem of intermittent points. Thus we will be
able to assume that a $p$-geodesic consists of a finite number of segments
each of which is either a geodesic (possibly with interior points where it
grazes the boundary) or a geodesic on the boundary (i.e. a segment of the
boundary in the two dimensional case).

We will show (Lemma \ref{localmin}) that the local minimizing property of
geodesics is inherited by $p$-geodesics for two dimensional manifolds with
boundary and no conjugate points.. The main problem is that the distance
function (and hence the energy function) is not $C^2$ even for the distance
between interior points (since the minimizing path can run along the
boundary for part of the time). Thus we need to be careful making second
variation arguments. On the other hand, we will be able to reduce to a
finite dimensional case (using piecewise $p$-geodesics) since it was shown
in \cite{ABB1} that for any compact $K\subset M$ there is a $b$ such that if
$p\in K$ and $q\in K$ have $d(p,q)\leq b$ then there is a unique minimizing $%
p$-geodesic between $p$ and $q$ and it varies continuously with $p$ and $q$.
Call $b$ the uniqueness radius of $K$. Thus if $\gamma:[0,1]\to M$ is any
path in $K$ and $0=t_0<t_1<t_2<...<t_{k-1}<t_k=1$ is a partition such that $%
d(\gamma(t_i),\gamma(t_{i+1}))\leq b$ then replacing $\gamma|[t_i,t_{i+1}]$
with the $p$-geodesic from $\gamma(t_i)$ to $\gamma(t_{i+1})$ (parameterized
on $[t_i,t_{i+1}]$) yields a piecewise $p$-geodesic curve with at most the
same energy.

We will consider the space of piecewise $C^1$ curves $\gamma:[0,1]\to M$
between two fixed points $x$ and $y$ in $M$ such that $E(\gamma)\leq E$ for
some fixed $E$. Since the length of any such curve is less than or equal to $E^{1/2}$,
all such curves lie in the closed (hence compact) ball $B(x,E^{1/2})$ of
radius $E^{1/2}$. We let $b$ be the uniqueness radius of that ball. Now for
any partition $0=t_0<t_1<t_2<...<t_{k-1}<t_k=1$ such that $t_{i+1}-t_i<%
\frac {b^2}{E}$ and any such $\gamma$, $L(\gamma|_{[t_i,t_{i+1}]})=%
\int_{t_i}^{t_{i+1}}|\gamma^{\prime}(t)|dt\leq
\{\int_{t_i}^{t_{i+1}}|\gamma^{\prime}(t)|^2dt%
\}^{1/2}(t_{i+1}-t_i)^{1/2}<E^{1/2}\frac b {E^{1/2}}=b$.

On The space $M^{k+1}$ (the product of $k+1$ copies of $M$) for each
partition $0=t_0<t_1<t_2<...<t_{k-1}<t_k=1$ we define the energy function $%
E_{(t_0,t_1,t_2,...,t_k)}:M^{k+1}\to R$ by
\begin{equation*}
E_{(t_0,t_1,t_2,...,t_k)}(x_0,x_1,...,x_k)={\sum}_{i=0}^k \frac {%
d(x_i,x_{i+1})^2}{t_{i+1}-t_i}.
\end{equation*}
This is defined so that $\gamma$ the piecewise $p$-geodesic curve defined by
$(x_0,x_1,...,x_k)$ satisfies $E(%
\gamma)=E_{(t_0,t_1,t_2,...,t_k)}(x_0,x_1,...,x_k)$. To be precise $\gamma$
is built of minimizing $p$-geodesics from $x_i$ to $x_{i+1}$ parameterized
proportional to arclength on $[t_1,t_{i+1}]$. Of course $%
E_{(t_0,t_1,t_2,...,t_k)}(\gamma(t_0),\gamma(t_1),\gamma(t_2),...,%
\gamma(t_k))\leq E(\gamma)$ for any curve $\gamma$.

The technical tool that will replace second variation arguments is

\begin{lem}
\label{localmin}

Let $M$ be a compact two dimensional manifold with smooth boundary, no
conjugate points and the boundary has finite $F_M$. Let $\gamma:[0,1]\to M$ be
$p$-geodesic from $x$ to $y$ of length $L$ parameterized proportional to
arclength, $E>L^2$, and $b$ defined as above. Then for any partition $%
0=t_0<t_1<t_2<...<t_{k-1}<t_k=1$ with $t_{i+1}-t_i<\frac {b^2}{E}$ there is
a neighborhood $U\subset M^{k+1}$ of $(\gamma(t_0),\gamma(t_1),%
\gamma(t_2),...,\gamma(t_k))\in M^{k+1}$ such that for any $u\in U$ with $u\neq (\gamma(t_0),\gamma(t_1),%
\gamma(t_2),...,\gamma(t_k))$ we have
$$
E_{(t_0,t_1,t_2,...,t_k)}(\gamma(t_0),\gamma(t_1),\gamma(t_2),...,%
\gamma(t_k))<E_{(t_0,t_1,t_2,...,t_k)}(u).$$
\end{lem}

\begin{proof}
The condition on $F_M$ tells us that $\gamma$ consists of a finite number of
segments each of which is either a geodesic with only endpoints on the boundary, a geodesic segment that lies on the boundary,
or a segment of the boundary where the boundary is strictly concave.  We can assume that $\gamma(0)$ and $\gamma(1)$ are not in the finite set $F_M$ since we can handle that case by taking limits of the more general case.

We next point out that that if the proposition is true for any partition $%
0=t_0<t_1<t_2<...<t_{k-1}<t_k=1$ with $t_{i+1}-t_i<\frac {b^2}{E}$ then it
is true for all such partitions. To see this let $%
0=s_0<s_1<s_2<...<s_{l-1}<s_l=1$ be such a partition where the proposition
does not hold. That means there is a sequence $u^i\in M^{l+1}$ representing
piecewise $p$-geodesic paths $\gamma^i$ from $x$ to $y$ which converges to $%
(\gamma(s_0),\gamma(s_1),\gamma(s_2),...,\gamma(s_l))$ such that $%
E_{(s_0,s_1,s_2,...,s_l)}(u^i)\leq E(\gamma)$. This means that $%
E(\gamma^i)\leq E(\gamma)$. Thus $E_{(t_0,t_1,t_2,...,t_k)}(\gamma^i(t_0),%
\gamma^i(t_1),\gamma^i(t_2),...,\gamma^i(t_k))\leq E(\gamma^i)\leq E(\gamma)$
while
$(\gamma^i(t_0),\gamma^i(t_1),\gamma^i(t_2),...,\gamma^i(t_k))$ converges to
$(\gamma(t_0),\gamma(t_1),\gamma(t_2),...,\gamma(t_k))$ which says that the
proposition does not hold for the partition $0=t_0<t_1<t_2<...<t_{k-1}<t_k=1$
either.

The previous paragraph of the proof now allows us to choose our $t_i$ such that $t_i\notin F_M$, thus one of three things hold.  In the first instance $\gamma(t_i)$ is an interior point of $M$. In this case we
call $t_i$ interior.  In the second, for all $t$ in some open interval about $t_i$, $%
\gamma(t)\in \partial M$ and is a geodesic (i.e. the geodesic curvature of $\partial M$ is $0$). In this case $t_i$ is called boundary geodesic. In the final case, for all $t$ in some open interval about $t_i$, $%
\gamma(t)\in \partial M$ and the boundary is strictly concave.  In this case we call $t_i$ boundary concave. The
condition on $F_M$ allows us to make sure that there is at least one $%
\gamma(t_i)$ for each concave boundary interval of $\gamma$. Consider a sequence $%
t_i, t_{i+1},...,t_{i+m}$ such that $m>1$, $t_i$ and $t_{i+m}$ are concave boundary
while $t_{i+1},...,t_{i+m-1}$ are interior or boundary geodesic.   Then the curve $%
\gamma|_{[t_i,t_{i+m}]}$ is precisely of the following form: there are
numbers $s_0$ and $s_1$ so that $t_i<s_0<t_{i+1}$, $t_{i+m-1}<s_1<t_{i+m}$, $%
\gamma|_{[s_0,s_1]}$ is a geodesic $\tau:[s_0,s_1]\to M$ while $%
\gamma|_{[t_i,s_0]}$ and $\gamma|_{[s_1,t_{i+m}]}$ are concave segments of the
boundary. By the concavity of the boundary at $\gamma(s_0)$ and $\gamma(s_1)$
there is an $\epsilon>0$ such that $\tau$ can be extended to a geodesic on $%
[s_0-\epsilon,s_1+\epsilon]$. Now any curve $\sigma:[t_i,t_i+m]\to M$ from $%
\gamma(t_i)$ to $\gamma(t_{i+m})$ close enough to $\gamma$ must intersect $%
\tau$ in at least two points $\sigma(a)=\tau(a^{\prime})$ and $%
\sigma(b)=\tau(b^{\prime})$ with $s_0-\epsilon<a^{\prime}\leq s_0$ and $%
s_1\leq b^{\prime}< s_1+\epsilon$. (See Figure \ref{fig:intersect2}.)
\begin{figure}[h]
  \labellist
  \small\hair 2pt
  \pinlabel $\gamma(t_i)$ [rb] at 41 114
  \pinlabel $\gamma(t_{i+m})$ [l] at 136 125
  \pinlabel $\tau$ [t] at 83 124
  \pinlabel $\sigma$ [b] at 84 135
  \endlabellist
  \center
  \includegraphics[width=0.4\textwidth]{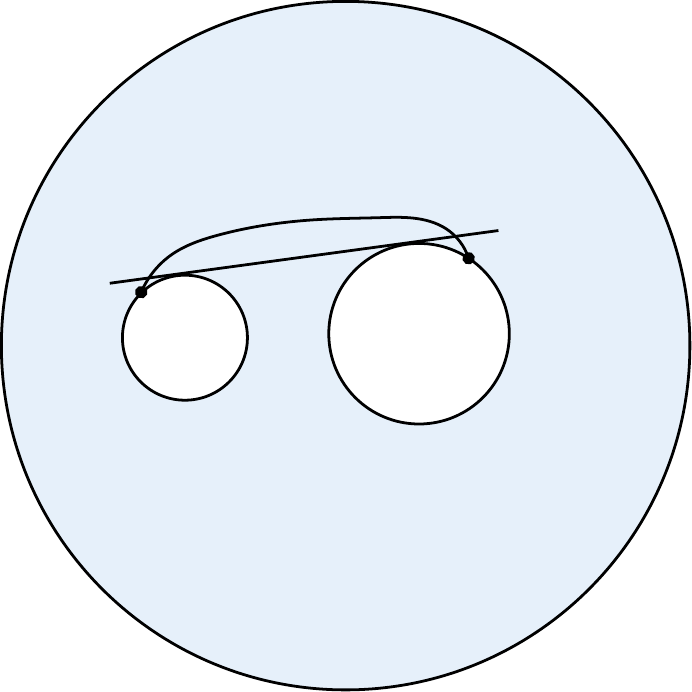}
  \caption{Any curve $\sigma$ from $\gamma(t_i)$ to $\gamma(t_{i+m})$ close enough to $\gamma$ must intersect $%
\tau$ in at least two points.}
  \label{fig:intersect2}
\end{figure}

Note that the unique local minimality of $%
\tau$, of $\gamma|_{[t_i,s_0]}$, and of $\gamma|_{[s_1,t_{i+m}]}$ tell us
that for $\sigma$ close enough to $\gamma$
\begin{equation*}
L(\sigma|_{[t_i,t_{i+m}]})\geq L(\sigma|_{[t_i,a]}\cup \tau|{%
[a^{\prime},b^{\prime}]} \cup \sigma|_{[b,t_{i+m}]})=
\end{equation*}
\begin{equation*}
= L(\sigma|_{[t_i,a]}\cup \tau|{[a^{\prime},s_0]} \cup \tau|{[s_0,s_1]} \cup
\tau|{[s_1,b^{\prime}]} \cup\sigma|_{[b,t_{i+m}]}))\geq
\end{equation*}
\begin{equation*}
\geq L(\gamma|{[t_i,s_0]}\cup \tau|{[s_0,s_1]}\cup \gamma|{[s_1,t_{i+1}]}%
)=L(\gamma|{[t_i,t_{i+1}]}),
\end{equation*}
and that equality can only hold if $\sigma|_{[t_i,t_{i+1}]}$ coincides with $%
\gamma|_{[t_i,t_{i+1}]}$ when parameterized proportional to arclength. Hence
$\gamma|_{[t_i,t_{i+m}]}$ is a strict local minimum of length for paths
between its endpoints. A similar (slightly easier) argument works when $i=0$
(respectively $i+m=k$) and $t_0$ (respectively $t_k$) is interior or boundary geodesic.

Assume that $\gamma$ does not satisfy the statement of the Lemma. Then there
is a sequence $u^j\in M^{k+1}$ never equal to $(\gamma(t_0),\gamma(t_1),%
\gamma(t_2),...,\gamma(t_k))$ but converging to it with corresponding
piecewise geodesics $\gamma^j$ converging to $\gamma$ and with $%
E(\gamma^j)\leq E_{(t_0,t_1,t_2,...,t_k)}(u^i)\leq
E_{(t_0,t_1,t_2,...,t_k)}(\gamma)=E(\gamma)$. In particular $L(\gamma^j)\leq
L(\gamma)$. Thus if $\bar \gamma^i$ is the reparametrization of $\gamma^i$
proportional to arclength then $w^j=(w_0^j,w_1^j,...,w_k^j)\equiv (\bar
\gamma^j(t_0),\bar \gamma^j(t_1),...,\bar \gamma^j(t_k))$ also converges to $%
(\gamma(t_0),\gamma(t_1),\gamma(t_2),...,\gamma(t_k))$ and for each $i$, $%
d(w^j_i,w^j_{i+1})\leq L(\gamma|_{[t_i,t_{i+1}]})$. Further if $t_i$ is
concave boundary then the strict concavity of the boundary near $\gamma(t_i)$ says
that for large $j$ we can replace $w^j_i$ with a point on the boundary so as
to decrease both $d(w^j_{i-1},w^j_i)$ and $d(w^j_i,w^j_{i+1})$ (unless $%
w^j_i $ lies on the boundary to begin with). Thus we will assume that $w^j_i$
lies on the boundary when $t_i$ is concave boundary. Define $\epsilon^j_i$ by $%
w^j_i=\gamma(t_i+\epsilon^j_i)$ which is well defined for all large $j$. By
the previous paragraph if $t_i$ and $t_{i+m}$ are consecutive concave boundary then $%
\epsilon^j_i-\epsilon^i_{i+m}\geq 0$, while if $t_{i_0}$ is the first concave
boundary and $t_{i_1}$ the last concave boundary we see $\epsilon^j_{i_0}\leq 0$ and
$\epsilon^j_{i_1}\geq 0$. This implies that for all concave boundary $t_i$, $%
\epsilon^j_i=0$ and $w^j_i=\gamma(t_i)$. Now the uniqueness of the previous
paragraph forces $w^j_i=\gamma(t_i)$ for all $i$. Since each step in moving
from $u_i$ to $w_i$ strictly decreased energy (unless no change was made) we
see that the original $u^j$ had to be simply $u^j_i=\gamma(t_i)$ yielding
the Lemma. \
\end{proof}

\bigskip

\begin{proof}
(of Proposition \ref{minimizing})

By passing to the universal cover we can assume that $M$ is simply connected
and that $p$-geodesic segments satisfy Lemma \ref{localmin}. We will
show that p-geodesic segments minimize. The idea is to mimic a standard
minimax argument using Lemma \ref{localmin} in place of saying that
all $p$-geodesics are critical points of index 0 for $E$.

Let $\gamma:[0,1]\to M$ be a p-geodesic segment (parameterized proportional to
arclength) from $x$ to $y$ and let $\tau:[0,1]\to M$ be a length minimizing $%
p$-geodesic from $x$ to $y$. We may assume that $E(\gamma) =
L^2(\gamma)>L^2(\tau)=E(\tau)$. By assumption there is a homotopy from $%
\gamma$ to $\tau$ in the space $\mathcal{C}_{(x,y)}$ of $C^1$ rectifiable
curves from $x$ to $y$. If $E$ is the maximum energy of a curve in this
homotopy and $b$ is the uniqueness radius of the closed ball of radius $%
E^{1/2}$ then by using the partition with $t_i=\frac i k$ where $\frac 1 k<%
\frac {b^2}{E}$ we can replace each of the curves in the homotopy with
piecewise $p$-geodesics (each piece of length less than $b$ parameterized on
an interval of length $\frac 1 k$). This defines a curve $u(s):[0,1] \to
M^{k+1}$ from $(x=\gamma(0),\gamma(\frac 1 k), \gamma(\frac 2
k),....,\gamma(1)=y)$ to $(x=\tau(0),\tau(\frac 1 k),\tau(\frac 2
k),....,\tau(1)=y)$ such that $E_{0,\frac 1 k,\frac 2 k,...,1}(u(s))\leq E$.
In fact $u(s)$ lies in the compact space $B(x,E^{1/2})^{k+1}$. We can take
the neighborhood $U(\gamma)$ of $(x=\gamma(0),\gamma(\frac 1 k),
\gamma(\frac 2 k),....,\gamma(1)=y)$ promised by Lemma \ref{localmin}
to be a small metric ball in $M^{k+1}$ (in the product metric) since the
boundary $\partial U(\gamma)$ is compact $E_{0,\frac 1 k,\frac 2
k,...,1}(u)\geq L^2(\gamma)+\epsilon$ for some $\epsilon>0$ and all $u\in \partial U(\gamma)$.

We now consider Let $E_0=inf \{max \{E_{0,\frac 1 k,\frac 2
k,...,1}(u(s))|s\in [0,1]\}\}$ where the infimum is taken over the
collection of all such curves $u(s)$. Since any $u(s)$ must cross $\partial
U(\gamma)$ we see that $E_0\geq L(\gamma)^2+\epsilon$ By an earlier argument
we know that all such curves lie in the compact space $B(x,E^{1/2})^{k+1}$.
Usual compactness arguments show that there is a minimax $p$-geodesic $%
\sigma $ from $x$ to $y$. That is: \newline
i) $E(\sigma)=E_{0,\frac 1 k,\frac 2 k,...,1}(\sigma(0),\sigma(\frac 1 k),
\sigma(\frac 2 k),....,\sigma(1))=E_0$ and \newline
ii) there are sequences $u_i$ and $s_i$ with $u_i(s_i)\to \sigma$ and \newline
$E_{0,\frac 1 k,\frac 2 k,...,1}(u_i(s_i))\geq E_{0,\frac 1 k,\frac 2
k,...,1}(u_i(s))$ for all $s$.
\newline
The argument is the same as the usual one - we sketch it. By compactnees
there are convergent sequences as in i) and ii) converging to $\sigma$. The
only thing to check is that we can assume $\sigma$ is $p$-geodesic. If $%
\sigma$ is any piecewise $p$-geodesic curve that is not a $p$-geodesic then
it has a nonzero angle at some join then there is a tangent vector $V\in
T_\sigma M^{k+1}$ which can be extended smoothly in a neighborhood such that
$V(E)<0$ at all points in the neighborhood. Thus (as usual) if our sequence $%
u_i(s_i)$ of curves with maximum energies approaching $E_0$ has no $p$-geodesic
as a minimax point then we could "push the curves" down (using the
above vector fields) to energies below $E_0$ which contradicts the
definition of $E_0$.

The above contradicts Lemma \ref{localmin}. Choose our neighborhood $%
U(\sigma)$ to be a small metric ball centered at $(\sigma(0),\sigma(\frac 1
k), \sigma(\frac 2 k),....,\sigma(1))$ and as before there is an $\epsilon_1$
such that $E_{0,\frac 1 k,\frac 2 k,...,1}(u)\geq E_0+\epsilon_1$ for all $u$
in $\partial U(\sigma)$. However for all large $i$ and all $s$ $E_{0,\frac 1
k,\frac 2 k,...,1}(u_i(s))\leq E_{0,\frac 1 k,\frac 2 k,...,1}(u_i(s_i))\leq
E_0+\frac {\epsilon_1} 2$ but since the curve $u_i(s)$ must intersect $%
\partial U(\sigma)$ we get the desired contradiction.

Note that $\gamma$ is the unique length minimizing path between its endpoints in its homotopy class.  This follows since if $\tau$ is another such then $\tau$ is also a $p$-geodesic and the above minimax argument for paths from $\gamma$ to $\tau$ leads to the same contradiction.

\end{proof}

\begin{rem}
Proposition \ref{minimizing} is false in higher dimensions.

To see this let us first consider a metric $g_0$ on $R^2-(0,0)$ defined in
polar coordinates by $ds^2=dr^2+f^2(r)d\theta^2$ where $f:[0,\infty)\to R^+$
is a smooth function such that: \newline
a) $f(r)=\sinh(r)$ for $r\geq 1$, \newline
b) $f(0)=0$, $f^{\prime}(r)>0$ and $f^{\prime\prime}(r)\geq 0$, \newline
c) $f(r)= r/3$ for $r\leq 1/10$. \newline
It is straightforward to check that such an $f$ exists. For example, find a
smooth function $f^{\prime}$ on $[0,1]$ such that $f^{\prime}(r)=1/3$ for $%
r\in [0,1/10]$, $f^{\prime}(r)=\cosh(r)$ for $r\in [9/10,1]$, $f^{\prime}(r)$
is increasing, and $\int_0^1 f^{\prime}(r)=\sinh(1)$. Then define $f(r)$ as $%
\int_0^r f^{\prime}(t)dt$. Such an $f^{\prime}$ exists since $\frac 1
{30}+\frac 8 {10}\cosh(\frac 9 {10})\approx 1.179802441$ which is larger
than $\sinh(\frac 9 {10}) \approx 1.026516726$

Condition b) tells us that the metric has nonpositive curvature and hence no
conjugate points.

Condition c) tells us that $ds^2$ defines a flat (cone like) metric for $%
r\leq 1/10$ gotten by taking a sector of the flat disc of radius $1/10$
subtending an angle of $\frac {2\pi} 3$ and gluing the edge radii together.
Thus there is a geodesic segment $\tau$ that self intersects (e.g. the
geodesic that corresponds to the straight line between the points 1/20 along
the edge radii). Choose $\epsilon>0$ so that $\tau\subset R^2-B(\epsilon)$
where we let $B(r)$ represent the open ball of radius $r$ centered at $(0,0)$%
. So the metric on $R^2-B(\epsilon)$ has no conjugate points and geodesics
that do not minimize. Of course it is not simply connected.

We will consider a metric on $R^3 - U$ where we think of $R^3$ as $R\times
R^2$ parameterized by $x$, $r$, and $\theta$. The metric will be $%
dx^2+dr^2+sinh(r)d\theta^2$ (i.e. a line cross with the hyperbolic metric on
the plane) when $|x|>1$ and it will be $dx^2+dr^2+f^2(r)d\theta^2$ for $%
|x|\leq 1$. The open set $U=\{(x,r,\theta)| -1<x<1,\ and\ r<r(x)\}$ where $%
r(x)$ is a smooth positive function such that $r(0)=\epsilon$, $r(1)>1$ and $%
r(-1)>1$. In particular $U$ is homeomorphic to a 3-ball and hence $R^3-U$ is
simply connected. We note that condition a) tells us that $g_0$ is just the
hyperbolic metric when $r>1$ and hence the metric $g$ is smooth on $R^3-U$
and has no conjugate points since it has nonpositive curvature. The curve $%
\tau$ on the totally geodesic $\{0\}\times (R^2-B(\epsilon))$ is a self
intersecting geodesic in $g$ and hence not minimizing.

To make this example compact simply use a large closed ball in $R^3$ (with
an extra boundary component) rather than all of $R^3$.
\end{rem}

\begin{Ack}

The second author is grateful for the research support and hospitality provided by Max Planck Institute for Mathematics during the preparation of this article.

\end{Ack}

\bibliography{mybib}{}
\bibliographystyle{alphanum}
\end{document}